\documentclass[letterpaper,10pt]{article}

\usepackage[letterpaper,left=1in,right=1in,top=1.5in,bottom=1.5in]{geometry}

\newcommand{\myauthor}{Benjamin Antieau and Daniel Bragg}
\newcommand{\mytitle}{Derived invariants from topological Hochschild homology}

\title{\mytitle}

\author{Benjamin Antieau and Daniel Bragg}
\date{\today}
 
\usepackage[pdfstartview=FitH,
            pdfauthor={\myauthor},
            pdftitle={\mytitle},
            colorlinks,
            linkcolor=citation,
            citecolor=citation,
            urlcolor=e-mail,
            ]{hyperref}

\usepackage{amsmath}
\usepackage{amscd}
\usepackage{amsbsy}
\usepackage{amssymb}
\usepackage{verbatim}
\usepackage{eufrak}
\usepackage{eucal}
\usepackage{microtype}
\usepackage{hyperref}
\usepackage{mathrsfs}
\usepackage{amsthm}
\usepackage[all,cmtip]{xy}
\usepackage{tikz}
\usetikzlibrary{matrix,arrows}
\usepackage{dsfont}
\usepackage{tikz-cd}   
\usepackage{float}
\usepackage{fancyhdr}

\usepackage{rotating}
\newcommand*{\isor}[2]{\arrow[#1,"\rotatebox{90}{\(\sim\)}","#2"']}

\usepackage{makeidx}
\makeindex

\usepackage{color}
\definecolor{todo}{rgb}{1,0,0}
\definecolor{conditional}{rgb}{0,1,0}
\definecolor{e-mail}{rgb}{0,.40,.80}
\definecolor{reference}{rgb}{.20,.60,.22}
\definecolor{mrnumber}{rgb}{.80,.40,0}
\definecolor{citation}{rgb}{0,.40,.80}

\let\oldmarginpar\marginpar
\renewcommand\marginpar[1]{\-\oldmarginpar[\raggedleft\footnotesize #1]%
{\raggedright\footnotesize #1}}

\newcommand*{\defeq}{\stackrel{\text{def}}{=}}

\renewcommand{\mathscr}[1]{\mathcal{#1}}

\newcommand{\Cscr}{\mathcal{C}}
\newcommand{\Dscr}{\mathcal{D}}

\newcommand{\Oscr}{\mathcal{O}}
\newcommand{\Pscr}{\mathcal{P}}

\newcommand{\D}{\mathrm{D}}
\newcommand{\E}{\mathrm{E}}
\newcommand{\F}{\mathrm{F}}
\renewcommand{\H}{\mathrm{H}}
\newcommand{\K}{\mathrm{K}}

\newcommand{\R}{\mathrm{R}}

\newcommand{\CC}{\mathds{C}}

\newcommand{\FF}{\mathds{F}}

\newcommand{\PP}{\mathds{P}}
\newcommand{\QQ}{\mathds{Q}}

\renewcommand{\SS}{\mathds{S}}

\newcommand{\ZZ}{\mathds{Z}}

\newcommand{\gr}{\mathrm{gr}}

\newcommand{\crys}{\mathrm{crys}}

\newcommand{\cyc}{\mathrm{cyc}}

\newcommand{\dR}{\mathrm{dR}}
\newcommand{\dRW}{\mathrm{dRW}}

\newcommand{\cris}{\mathrm{crys}}

\newcommand{\CycSp}{\mathbf{CycSp}}

\newcommand{\un}{\mathrm{un}}

\DeclareMathOperator{\dlog}{dlog}

\newcommand{\heart}{\heartsuit}

\DeclareMathOperator{\tors}{tors}

\newcommand{\im}{\mathrm{im}}

\renewcommand{\geq}{\geqslant}
\renewcommand{\leq}{\leqslant}

\newcommand{\Cart}{\mathrm{Cart}}
\newcommand{\Dieu}{\mathrm{Dieu}}

\DeclareMathOperator{\Ext}{Ext}
\newcommand{\HC}{\mathrm{HC}}
\newcommand{\THH}{\mathrm{THH}}
\newcommand{\HP}{\mathrm{HP}}
\newcommand{\TP}{\mathrm{TP}}

\newcommand{\TR}{\mathrm{TR}}
\newcommand{\HH}{\mathrm{HH}}

\DeclareMathOperator{\Pic}{Pic}
\DeclareMathOperator{\Br}{Br}

\DeclareMathOperator{\Dom}{Dom}

\newcommand{\Perfscr}{\Pscr\mathrm{erf}}

\newcommand{\Gm}{\mathds{G}_{m}}

\newcommand{\et}{\mathrm{\acute{e}t}}

\DeclareMathOperator{\Spec}{Spec}

\DeclareMathOperator{\red}{red}

\newcommand{\we}{\simeq}
\newcommand{\iso}{\cong}

\theoremstyle{plain}
\newtheorem{theorem}{Theorem}[section]
\newtheorem*{theorem*}{Theorem}
\newtheorem{lemma}[theorem]{Lemma}

\newtheorem{proposition}[theorem]{Proposition}

\newtheorem{corollary}[theorem]{Corollary}
\newtheorem*{corollary*}{Corollary}

\theoremstyle{plain}

\theoremstyle{definition}

\newtheoremstyle{named}{}{}{\itshape}{}{\bfseries}{.}{.5em}{#1 \thmnote{#3}}
\theoremstyle{named}

\theoremstyle{definition}
\newtheorem{definition}[theorem]{Definition}

\newtheorem{notation}[theorem]{Notation}

\newtheorem{example}[theorem]{Example}
\newtheorem*{example*}{Example}

\newtheorem{question}[theorem]{Question}
\newtheorem*{question*}{Question}

\newtheorem{remark}[theorem]{Remark}

\begin{document}

\maketitle

\begin{abstract}
    \noindent
    We consider derived invariants of varieties in positive characteristic
    arising from topological Hochschild homology. Using theory developed by Ekedahl and Illusie--Raynaud in their study of the slope spectral sequence, we examine the behavior under derived equivalences of various $p$-adic quantities related
    to Hodge--Witt and crystalline cohomology groups, including slope
    numbers, domino numbers, and Hodge--Witt numbers. As a consequence, we obtain restrictions on the Hodge numbers of derived equivalent varieties, partially extending results of Popa--Schell to positive characteristic.
    

    \paragraph{Mathematics Subject Classification 2010.}
    \href{http://www.ams.org/mathscinet/msc/msc2010.html?t=14Fxx&btn=Current}{14F30}
    --
    \href{http://www.ams.org/mathscinet/msc/msc2010.html?t=14Fxx&btn=Current}{14F40}
    --
    \href{http://www.ams.org/mathscinet/msc/msc2010.html?t=19Dxx&btn=Current}{19D55}.
\end{abstract}


\section{Introduction}

\newcommand{\BMS}{\mathrm{BMS}}
\newcommand{\DF}{\mathrm{DF}}
\newcommand{\HKR}{\mathrm{HKR}}
\newcommand{\dHH}{\mathrm{dHH}}
\newcommand{\dHC}{\mathrm{dHC}}
\newcommand{\dHP}{\mathrm{dHP}}
\newcommand{\DBF}{\mathrm{DBF}}
\newcommand{\CW}{\mathrm{CW}}

In this paper we study derived invariants of varieties in positive characteristic.

    \begin{itemize}
        \item Let $X$ and $Y$ be smooth proper $k$-schemes for some field
            $k$. We say that $X$ and $Y$ are {\bf Fourier--Mukai equivalent}, or {\bf
            FM-equivalent}, if there is a complex $P\in\Dscr^b(X\times_k Y)$
            such that the induced functor
            $\Phi_P\colon\Dscr^b(X)\rightarrow\Dscr^b(Y)$ is an equivalence,
            where $\Dscr^b(-)$ denotes the dg category of bounded complexes of
            coherent sheaves.
            This is equivalent to asking for $\Dscr^b(X)$ and $\Dscr^b(Y)$ to
            be equivalent as $k$-linear dg categories. When $X$ and $Y$ are
            smooth and projective, they are FM equivalent if and only if there
            is a $k$-linear triangulated equivalence $\D^b(X)\we\D^b(Y)$ by
            Orlov's theorem (see~\cite[Theorem~5.14]{huybrechts}).
        \item Let $h$ be a numerical or categorical invariant of smooth proper
            $k$-schemes. We say that $h$
            is a {\bf derived invariant} if whenever there is a Fourier--Mukai
            equivalence $\Dscr^b(X)\we\Dscr^b(Y)$ we have $h(X)= h(Y)$.
    \end{itemize}

The Hochschild homology groups $\HH_*(X/k)$ of a smooth
proper variety $X$, which are finite dimensional vector spaces over $k$, 
are derived invariants. In characteristic $0$, and in
characteristic $p$ for $p\geq d=\dim X$, the Hochschild--Kostant--Rosenberg
isomorphism~\cite{hkr} relates the
Hochschild homology groups to Hodge cohomology groups $\H^*(X,\Omega^*_X)$. We
briefly review this story in Section \ref{sec:Hochschild homology}. This
relationship has been extensively studied, and plays a key role in our
understanding of derived categories of varieties, especially over the complex
numbers.

Suppose now that $k$ is a perfect field of characteristic $p>0$. To a smooth
proper variety $X$ over $k$ we may associate via topological methods certain $p$-adic
analogs of Hochschild homology: the topological Hochschild homology groups
$\THH_*(X)$, as well as the related groups $\TR_*(X)$ and $\TP_*(X)$. We recall
this theory in Section \ref{sec:Topological Hochschild homology}. These are
modules over $W=W(k)$ (the ring of Witt vectors of $k$) which are equipped with certain extra semilinear structures, and whose construction moreover
depends only on $\Dscr^b(X)$. Furthermore, by a result of
Hesselholt~\cite{Hesselholt}, the
$\TR_*(X)$ may be computed in terms of the Hodge--Witt cohomology groups
$\H^*(W\Omega^*_X)$, similar to the classical relationship between Hochschild homology and Hodge cohomology. Our goal in this paper is to study these objects as derived invariants of the variety $X$.

Our key technical results are obtained in
Section \ref{sec:compatible}, where we analyze the spectral sequence connecting
Hodge--Witt cohomology groups to $\TR_*(X)$, and study the extent to which the extra structures of $F,V,d$ are preserved. We also recall the theory of coherent $R$-modules introduced by Illusie and Raynaud \cite{illusie-raynaud} in their study of the slope spectral sequence, and explain how our noncommutative structures fit into this framework.

For the remainder of the paper we study consequences of this theory. In Section \ref{sec:derived invariants from THH}, we observe the following fact, which extends
results of Bragg--Lieblich~\cite{bragg-lieblich-twistor}. Given a smooth proper
$d$-dimensional $k$-scheme, following Artin--Mazur~\cite{artin-mazur}, we let $\Phi^d_X$ be the functor on augmented Artin
local $k$-algebras
defined by $\Phi^d_X(A)=\ker\left(\H^d(X\times_{\Spec k}\Spec
A,\Gm)\rightarrow\H^d(X,\Gm)\right)$. 

\begin{proposition}\label{prop:formal groups}
    Let $X$ and $Y$ be Calabi--Yau $d$-folds over a perfect field $k$ of
    positive characteristic. If $X$ and $Y$ are FM-equivalent, then $\Phi^d_X\iso\Phi^d_Y$. In particular, the heights of $X$ and $Y$ are equal.
\end{proposition}

When $d=2$, one calls $\Phi^2_X$ the {\bf formal Brauer group} of $X$.
Proposition~\ref{prop:formal groups} implies in particular that the height of a
K3 surface is a derived invariant, which was already known. We then study
derived invariants of surfaces in more detail. We find another proof that the
Artin invariant is a derived invariant of supersingular K3 surfaces, and
recover a result of Tirabassi on Enriques surfaces.

In Sections \ref{sec:slopes and isogeny invariants}, \ref{sec:domino numbers},
and \ref{sec:Hodge-Witt numbers}, we introduce various numerical $p$-adic
invariants. Specializing to the case of varieties of dimension $d\leq 3$, we
prove the following (for definitions, see the body of the paper).

\begin{theorem}\label{thm:omnibus}
    Let $k$ be a perfect field of positive characteristic $p$, let $W=W(k)$ be the ring
    of $p$-typical Witt vectors over $k$, and let $K=W[p^{-1}]$ be the
    fraction field of $W$.
    The following are derived invariants of smooth proper varieties of
    dimension $d\leq 3$ over $k$:
    \begin{enumerate}
        \item[{\rm (1)}] the slopes of Frobenius with multiplicity acting on the rational Hodge--Witt cohomology groups
            $\H^j(W\Omega^i_X)\otimes_W K$ and rational crystalline cohomology groups $\H^i(X/K)$,
        \item[{\rm (2)}] the domino numbers $T^{i,j}$,
        \item[{\rm (3)}] the Hodge--Witt numbers $h^{i,j}_W$,
        \item[{\rm (4)}] the Zeta function $\zeta(X)$ (if $X$ is defined over a finite
            field), and
        \item[{\rm (5)}] the Betti numbers $b_i=\dim_K\H^i(X/K)$.
    \end{enumerate}
\end{theorem}

Part (4) was previously proved by Honigs in~\cite{honigs-3}, whose methods also suffice to prove (1) and (5). The proofs of statements (2) and (3), however, crucially rely on the topological derived invariants.

In Section \ref{sec:Hodge numbers} we consider the question of whether the
Hodge numbers $h^{i,j}=\dim_k\H^j(X,\Omega^i_X)$ are derived invariants. For context, we note that a conjecture of
Orlov~\cite{orlov} states that the rational Chow motive $\mathfrak{h}_X$ of a variety $X$ is
a derived invariant. In characteristic $0$, the Hodge numbers are determined by
the rational Chow motive, and so Orlov's conjecture implies that the Hodge
numbers are derived invariant. This consequence has been verified by
Popa--Schnell \cite{popa-schnell} for varieties of dimension $d\leq 3$. However,
as discussed in Section \ref{sec:Hodge numbers}, their proof breaks in several
ways in positive characteristic. In characteristic $p$, the Hodge numbers are
related (in a somewhat subtle way) to Hodge--Witt cohomology groups. Using this relationship and Theorem \ref{thm:omnibus}, we prove the following.
\begin{theorem}
    Suppose that $X$ and $Y$ are FM-equivalent smooth proper varieties of
    dimension $d$ over a field $k$ of positive characteristic.
    \begin{enumerate}
        \item[{\rm (1)}] If $d\leq 2$, then $h^{i,j}(X)=h^{i,j}(Y)$ for all $i,j$.
        \item[{\rm (2)}] If $d\leq 3$, then $\chi(\Omega^i_X)=\chi(\Omega^i_Y)$ for all $i$.
    \end{enumerate}
\end{theorem}
We remark that this result uses topological Hochschild homology constructions in a key way. Even in the case of surfaces, we do not know a direct proof using only Hochschild homology.

Under a mild additional assumption, we are able to strengthen this result for $d=3$.
We say that a smooth proper variety $X$ over a perfect field $k$ of positive
characteristic is {\bf Mazur--Ogus} if the Hodge--de Rham
spectral sequence for $X$ degenerates at $\E_1$ and the crystalline cohomology
groups of $X$ are torsion-free. The class of Mazur--Ogus schemes includes
smooth complete intersections, abelian varieties, and K3 surfaces.

\begin{theorem}
    Suppose that $X$ and $Y$ are FM-equivalent smooth proper varieties of
    dimension $d=3$ over a perfect field $k$ of positive characteristic $p\geq 3$. If
    $X$ is Mazur--Ogus, then so is $Y$, and $h^{i,j}(X)=h^{i,j}(Y)$ for all
    $i,j$.
\end{theorem}

Finally, in Section~\ref{sec:K3}, we compute $\TR$ and $\TP$ for
twisted K3 surfaces and discuss how to recover the fine structure of the Mukai
lattice from $\TP$.

\paragraph{Conventions.}
We will use many spectral sequences in this paper. They all converge for smooth
and proper schemes or dg categories, so we will say
nothing more about convergence in our discussion.
  
If $X$ is a smooth proper scheme
over $k$, we will write $\H^*(X/W)$ for the crystalline cohomology groups of $X$ relative to $W=W(k)$ and $\H^*(W\Omega^*_X)=\H^*(X,W\Omega^*_X)$ for the Hodge--Witt cohomology groups of $X$, as defined in \cite{illusie-derham-witt}.

\paragraph{Acknowledgments.} The first author was supported by NSF Grant DMS-1552766. The second author was partially supported by NSF RTG grant DMS-1646385 and by NSF postdoctoral fellowship DMS-1902875. Both authors were supported by the National Science Foundation under Grant DMS-1440140 while in residence at the
Mathematical Sciences Research Institute in Berkeley, California during the
Spring 2019 semester.

\section{Hochschild homology}\label{sec:Hochschild homology}

Let $k$ be a commutative ring.
For any $k$-linear dg category $\Cscr$, the Hochschild homology of $\Cscr$ over
$k$ is an object
$\HH(\Cscr/k)\in\Dscr(k)$ which is equipped with an action of the circle $S^1$. The
homology groups $\H_i(\HH(\Cscr/k))=\HH_i(\Cscr/k)$ are the Hochschild homology
groups of $\Cscr$ over $k$. For a scheme $X/k$, we let
$\HH(X/k)=\HH(\Perfscr(X)/k)$, where $\Perfscr(X)$ is the $k$-linear dg
category of perfect complexes on $X$. This is a {\em noncommutative invariant} of $k$-schemes meaning in
particular that if $\Perfscr(X)\we\Perfscr(Y)$, then $\HH(X/k)\we\HH(Y/k)$ as
complexes with $S^1$-action. For some details on Hochschild homology and the
constructions below from a classical perspective, see~\cite{loday}; for details
on Hochschild homology from a modern perspective, see~\cite{bms2}.

While the Hochschild homology of a smooth proper $k$-scheme $X$ is a derived invariant, one often computes it via the following spectral sequence, which is not.

\begin{definition}
  The {\bf Hochschild--Kostant--Rosenberg spectral sequence}
\begin{equation}\label{eq:HKR spectral sequence}
    \E_2^{s,t}=\H^t(X,\Omega^s_{X})\Rightarrow\HH_{s-t}(X/k),
\end{equation}
is the descent, or local-to-global,
spectral sequence for Hochschild homology.\footnote{With this indexing the
differentials $d_r$ have bidegree $(r-1,r)$; this convention has the advantage
that the $\E_2$ page of~\eqref{eq:HKR spectral sequence} agrees with the $\E_1$
page of the Hodge--de Rham spectral sequence~\eqref{eq:Hodge to de Rham}.} Here, $\Omega^*_X=\Omega^*_{X/k}$ are the sheaves of de Rham forms relative to $k$.
\end{definition}
The HKR spectral sequence is known to degenerate for smooth
schemes in characteristic zero, or more generally when $\dim(X)!$ is invertible in $k$.
It also degenerates in characteristic $p$ when $\dim(X)\leq p$
by~\cite{antieau-vezzosi}. In general, when $\dim(X)>p$, the HKR spectral sequence
does not degenerate; for examples,
see~\cite{antieau-bhatt-mathew}. If~\eqref{eq:HKR spectral sequence} degenerates, then there exist non-canonical isomorphisms
\[
  \HH_i(X/k)\cong\bigoplus_{j}\H^{j-i}(X,\Omega^j_X)
\]
of $k$-vector spaces for each $i$. The above discussion implies the following well known result.

\begin{theorem}\label{thm:HH and hodge numbers}
    Let $X$ and $Y$ be FM-equivalent smooth proper schemes of dimension $d$ over a field $k$ of characteristic $p$. If $p=0$ or $p\geq d$, there exist isomorphisms
    \[
      \bigoplus_{j}\H^{j-i}(X,\Omega^j_X)\cong
      \bigoplus_{j}\H^{j-i}(Y,\Omega^j_Y).
    \]
    In particular,
    \[
      \sum_j h^{j,j-i}(X)=\sum_j h^{j,j-i}(Y)
    \]
    for all $i$.
\end{theorem}

From Hochschild homology, one constructs several other noncommutative invariants, namely the cyclic
homology $\HC(\Cscr/k)=\HH(\Cscr/k)_{hS^1}$ obtained using the $S^1$-homotopy
orbits, the negative cyclic homology
$\HC^-(\Cscr/k)=\HH(\Cscr/k)^{hS^1}$ obtained using the $S^1$-homotopy fixed
points,
and the periodic cyclic homology $\HP(\Cscr/k)=\HH(\Cscr/k)^{tS^1}$ obtained
using the $S^1$-Tate construction. See~\cite{loday} background. For a $k$-scheme $X$, each of
these theories is computed by two spectral sequences: a noncommutative spectral
sequence and a de Rham spectral sequence. We review the theory for $\HP(X/k)$;
the other cases are similar.

\begin{definition}
    By definition of the Tate construction (see for example~\cite{nikolaus-scholze}),
    there is a {\bf Tate spectral sequence}
    \begin{equation}\label{eq:Tate to HP}
        \E_2^{s,t}=\widehat{\H}^s(\CC\PP^\infty,\HH_t(X/k))\Rightarrow\HP_{t-s}(X/k)
    \end{equation}
    computing $\HP(X/k)$,
    with differentials $d_r$ of bidegree $(r,r-1)$,
    where $\widehat{\H}^\ast(\CC\PP^\infty,-)$ is a $2$-periodic version of the
    cohomology of $\CC\PP^\infty$. When computing $\HP_*(X/k)$ via a mixed
    complex as in~\cite{loday}, this is the spectral sequence arising from the
    filtration by columns. This is often called the noncommutative
    Hodge--de Rham spectral sequence.
\end{definition}

\begin{definition}
    Let $X$ be a smooth and proper $k$-scheme.
    There is a {\bf de Rham--$\mathbf{HP}$ spectral sequence}
    \begin{equation}\label{eq:de Rham to HP}
        \E_2^{s,t}=\H^{s-t}_{\dR}(X/k)\Rightarrow\HP_{-s-t}(X/k),
    \end{equation}
    with differentials $d_r$ of bidegree $(r,1-r)$,
    where $\R\Gamma_{\dR}(X/k)$ is the de Rham cohomology of $X/k$ and
    $\H^{s-t}_{\dR}(X/k)=\H^{s-t}(\R\Gamma_{\dR}(X/k))$. This spectral sequence was constructed in~\cite{bms2} in the
    $p$-adically complete situation and in~\cite{antieau-derham} in general.
    In characteristic zero, it can easily be extracted
    from~\cite{toen-vezzosi-simpliciales}.
\end{definition}

Finally, for a smooth proper $k$-scheme, we have the {\bf Hodge--de Rham spectral sequence}
    \begin{equation}\label{eq:Hodge to de Rham}
        \E_1^{s,t}=\H^t(X,\Omega^s_{X/k})\Rightarrow\H^{s+t}_{\dR}(X/k)
    \end{equation}
which has differentials $d_r$ of bidegree $(r,1-r)$. We summarize our situation
in Figure~\ref{fig:SS diamond 1}.

\begin{figure}[H]
  \centering
  \begin{tikzcd}
    &\HH_*(X/k)\arrow[Rightarrow]{dr}{\text{Tate}}&\\
    \H^*(X,\Omega^*_X)\arrow[Rightarrow]{ur}{\text{HKR}}\arrow[Rightarrow]{dr}[swap]{\text{Hodge--de Rham}}&&\HP_*(X/k)\\
    &\H^*_{\dR}(X/k)\arrow[Rightarrow]{ur}[swap]{\text{de Rham--HP}}&
  \end{tikzcd}
  \caption{Four spectral sequences associated to a smooth proper scheme $X$ over $k$.}
  \label{fig:SS diamond 1}
\end{figure}
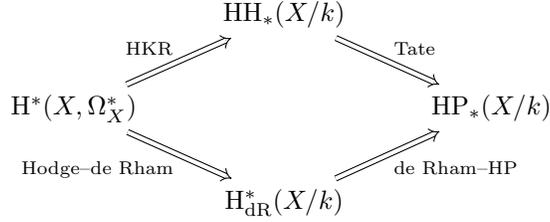

The Tate spectral sequence~\eqref{eq:Tate to HP} itself, meaning the collection
of pages and differentials, is a derived invariant. However, there is no reason
for the de Rham--$\HP$ spectral
sequence~\eqref{eq:de Rham to HP} to be derived invariant, although the objects it computes are derived invariants. 

\begin{remark}
    Playing these spectral sequences off of each other can be profitable. For
    example, if $k$ is a field and if $X/k$ is smooth and proper and if the HKR spectral sequence~\eqref{eq:HKR spectral sequence} degenerates (for example if
    $\dim(X)\leq p$), then the degeneration of the Tate spectral
    sequence~\eqref{eq:Tate to HP} implies degeneration of the de Rham--$\HP$ spectral sequence~\eqref{eq:de Rham to HP} {\em and}
    the Hodge--de Rham spectral sequence~\eqref{eq:Hodge to de Rham}. Similarly, if the HKR and Hodge--de Rham spectral sequences degenerate,
    then the Tate spectral sequence degenerates if and only if the de Rham--$\HP$ spectral sequence degenerates.
\end{remark}

If $k$ is a perfect field,
the Tate spectral sequence~\eqref{eq:Tate to HP} computing $\HP$ degenerates when $\Cscr$ is smooth and
proper over $k$, $\HH_i(\Cscr/k)=0$ for $i\notin[-p,p]$, and $\Cscr$ lifts to $W_2(k)$ by
work of Kaledin~\cite{kaledin-nhdr,kaledin-spectral} (see
also Mathew's paper~\cite{mathew-degeneration}).
Using this fact, we prove a theorem which implies Hodge--de Rham
degeneration in many cases.

\begin{theorem}\label{thm:lift}
    Let $k$ be a perfect field of positive characteristic $p$. Let $X,Y$ be smooth proper schemes such that
    $\Dscr^b(X)\we\Dscr^b(Y)$ and $\dim(X)=\dim(Y)\leq p$. If $X$ lifts to
    $W_2(k)$, then the Hodge--de Rham spectral sequence degenerates for $Y$.
\end{theorem}

\begin{proof}
    The HKR spectral sequence~\eqref{eq:HKR spectral sequence} degenerates for both $X$ and $Y$
    by~\cite{antieau-vezzosi}. Since $X$ lifts to $W_2(k)$, the Tate spectral
    sequence~\eqref{eq:Tate to HP} degenerates for $X$. This tells us the total
    dimension of $\HP(X/k)\we\HP(Y/k)$ and hence implies that the Tate
    spectral sequence~\eqref{eq:Tate to HP} degenerates for $Y$ as well. The existence of the
    convergent de Rham--$\HP$ spectral sequence~\eqref{eq:de Rham to HP} now implies that the Hodge--de
    Rham spectral sequence~\eqref{eq:Hodge to de Rham} degenerates for $Y$ by counting dimensions.
\end{proof}

The theorem is some evidence for a positive answer to the following question.

\begin{question}[Lieblich]\label{question:lieblich}
    Let $X$ and $Y$ be FM-equivalent smooth proper varieties over a
    perfect field $k$ of positive characteristic $p$. If $X$ lifts to characteristic 0 (or lifts to $W_2(k)$,
    etc.), does $Y$ also lift?
\end{question}

\section{Topological Hochschild homology}\label{sec:Topological Hochschild homology}

In this section, we introduce the main tools of this paper, which are
topological analogs of the invariants of the previous section. Here,
topological means that one works relative to the sphere spectrum $\SS$, which
is the initial commutative ring (spectrum) in homotopy theory. To any stable
$\infty$-category or dg category $\Cscr$, one can associate a spectrum
$\THH(\Cscr)=\HH(\Cscr/\SS)$ with $S^1$-action. This is again a noncommutative
invariant and there are various analogs of the spectral sequences of the
previous section. We will especially be interested in the {\bf topological periodic cyclic homology}
\[
  \TP(\Cscr)=\THH(\Cscr)^{tS^1}.
\]
Topological Hochschild homology $\THH(\Cscr)$ is equipped with an even richer
structure than simply an $S^1$-action: it is a cyclotomic spectrum, a notion
introduced by B\"okstedt--Hsiang--Madsen~\cite{bhm} to study algebraic
$K$-theory and recently recast by Nikolaus and
Scholze in~\cite{nikolaus-scholze}. We use the following definition, which is a slight alteration of the main
definition of~\cite{nikolaus-scholze}.

\begin{definition}
    A {\bf $p$-typical cyclotomic spectrum} is a spectrum $X$ with an
    $S^1$-action together with an $S^1$-equivariant map $X\rightarrow
    X^{tC_p}$, called the cyclotomic Frobenius, where $X^{tC_p}$ is equipped
    with the $S^1$-action coming from the isomorphism $S^1\iso S^1/C_p$.
    We let $\CycSp_p$ denote the stable $\infty$-category of $p$-typical cyclotomic
    spectra. If $k$ is a commutative ring, then $\THH(k)$ is a commutative
    algebra object of $\CycSp_p$ and we let $\CycSp_{\THH(k)}$ denote the
    stable $\infty$-category of $\THH(k)$-modules in $p$-typical cyclotomic
    spectra.
\end{definition}

The exact nature of a cyclotomic spectrum
will not concern us much, except in the extraction of homotopy objects with
respect to a natural $t$-structure on cyclotomic spectra studied in~\cite{antieau-nikolaus}.

\begin{definition}\label{def:dieudonne modules etc}
    A {\bf $p$-typical Cartier module} is an abelian group $M$ equipped with endomorphisms $F$
    and $V$ such that $FV=p$ on $M$. A {\bf Dieudonn\'e module} over a perfect
    field $k$ is a $W=W(k)$-module $M$ equipped with endomorphisms $F$ and $V$ satisfying $FV=VF=p$ and which are compatible with the Witt vector Frobenius $\sigma$ in the sense that
    \[
      F(am)=\sigma(a)F(m)\hspace{1cm}\mbox{and}\hspace{1cm} V(\sigma(a)m)=aV(m)
    \]
    for $a\in W$ and $m\in M$. Note that we do not require that $M$ is finitely generated or torsion free.
    
    A Cartier or Dieudonn\'e module $M$ is {\bf derived $V$-complete} if the
    natural map $M\rightarrow\R\lim_n M//V^n$ is an equivalence, where $M//V^n$
    is the cofiber of $V^n\colon M\rightarrow M$ in the derived category of
    abelian groups. Let $\widehat{\Cart}_p$ denote the abelian category of
    derived $V$-complete Cartier modules and let $\widehat{\Dieu}_k$ denote the
    abelian category of derived $V$-complete Dieudonn\'e modules over $k$.
\end{definition}

The full subcategory of $\CycSp_p$ of $p$-typical cyclotomic spectra $X$ such
that $\pi_iX=0$ for $i<0$ defines the connective part of a $t$-structure on
$\CycSp_p$ and similarly for $\CycSp_{\THH(k)}$. The main theorem
of~\cite{antieau-nikolaus} identifies the heart.

\begin{theorem}[\cite{antieau-nikolaus}]\label{thm:heart}
    Let $k$ be a perfect field of positive characteristic $p$.
    There are equivalences of abelian categories $\CycSp_p^\heart\we\widehat{\Cart}_p$ and
    $\CycSp_{\THH(k)}^\heart\we\widehat{\Dieu}_k$.
\end{theorem}

To a $p$-typical cyclotomic spectrum $X$, we can associate a new spectrum
$\TR(X)$ with $S^1$-action and with natural endomorphisms $F$ and $V$ making the homotopy groups of
$\TR(X)$ into Cartier modules. The construction of $\TR(X)$ was introduced by
Hesselholt in~\cite{Hesselholt}. It is proven
in~\cite{antieau-nikolaus} that $\pi_i\TR(X)\iso\pi_i^\cyc(X)$ as Cartier
modules under the equivalence of Theorem~\ref{thm:heart}, where $\pi_i^\cyc(X)$
denotes the $i$th homotopy object of $X$ with respect to the $t$-structure on
cyclotomic spectra. In the case of a
scheme $X$, we let $\TR(X)=\TR(\THH(X))$. If $k$ is a perfect field of positive characteristic $p$, then for
any cyclotomic spectrum $X\in\CycSp_{\THH(k)}$, the homotopy groups
$\TR_i(X)=\pi_i\TR(X)$ are equipped with differentials 
\[
  \TR_i(X)\xrightarrow{d}\TR_{i+1}(X)
\]
coming from the $S^1$-action making $\TR_*(X)$ into an $R$-module, where $R$ is the
Cartier--Dieudonn\'e--Raynaud ring; see Section~\ref{sec:compatible}.

\begin{example}
    When $k$ is a perfect field of positive characteristic $p$, $\THH(k)$ is in the
    heart $\CycSp_{\THH(k)}^\heart$ and corresponds to the ring of Witt vectors $W(k)$ with its Witt
    vector Frobenius and Vershiebung maps. In this language, the result is due
    to~\cite[Theorem~2]{antieau-nikolaus}, but the underlying computation that
    $\pi_*\TR(k)\iso W(k)$ is due to Hesselholt--Madsen~\cite[Theorem~5.5]{HM97}.
    The fact should be compared to the fundamental computation of B\"okstedt
    which says that $\pi_*\THH(k)\iso k[b]$ where $|b|=2$
    (see~\cite{bokstedt} for the case of $k=\FF_p$
    and~\cite[Corollary~5.5]{HM97} for the general case).
\end{example}
    
Let $A$ be a smooth commutative $k$-algebra where $k$ is a perfect field of
positive characteristic $p$. The
homotopy groups of $\THH(A)$ were computed in~\cite{Hesselholt}, where it
is shown that $\pi_*\THH(A)\iso\Omega^*_{A/k}\otimes_kk[b]$, where $|b|=2$ and
$\Omega^i_{A/k}$ lives in homological degree $i$. This is analogous to the HKR
isomorphism for Hochschild homology, but more complicated thanks to B\"okstedt's
class $b$. 

However, when working with $\TR$ instead, Hesselholt
proved in~\cite{Hesselholt} an exact de Rham--Witt analog of the HKR isomorphism: there is a
graded isomorphism $\TR_*(A)\iso W\Omega^*_A$ compatible with the $F$, $V$, and $d$
operations, where $W\Omega^\bullet_A$ is the de Rham--Witt complex of $A$ as
studied in~\cite{illusie-derham-witt}.\footnote{We will use $W\Omega^*_A$ to
denote the graded abelian group underlying the complex
$W\Omega^\bullet_A$.}

Now, we can define the topological or crystalline analogs of the four
spectral sequences from Section~\ref{sec:Hochschild homology}.

\begin{definition}\label{def:four spectral sequences}
Let $X$ be a smooth proper scheme over a perfect field $k$ of positive
characteristic $p$ and let $\Cscr$ be a smooth proper dg category over $k$.
    \begin{enumerate}
        \item[(a)] Using Hesselholt's local calculation, the {\bf descent
            spectral sequence} for $\TR$ is
            \begin{equation}\label{eq:Hesselholt SS}
                \E_2^{s,t}=\H^t(X,W\Omega^s_{X})\Rightarrow\TR_{s-t}(X).
            \end{equation} 
            With this indexing, the differentials $d_r$ have bidegree $(r-1,r)$.
            This is the topological analog of~\eqref{eq:HKR spectral sequence}.
        \item[(b)]    There is a {\bf Tate spectral sequence}
            \begin{equation}\label{eq:Tate for TP}
                \E_2^{s,t}=\widehat{\H}^s(\CC\PP^\infty,\TR_t(\Cscr))\Rightarrow\TP_{t-s}(\Cscr)
            \end{equation}
            computing $\TP(\Cscr)$,
            with differentials $d_r$ of bidegree $(r,r-1)$. This is the $\TP$
            analog of~\eqref{eq:Tate to HP}. It was constructed
            in~\cite[Corollary~10]{antieau-nikolaus}. See Figure~\ref{eq:Tate for
            TP visual}. By analogy with~\eqref{eq:Tate to HP}, this spectral
        sequence could be called the ``non-commutative slope spectral
    sequence''.
        \item[(c)] There is a {\bf
            crystalline--$\TP$ spectral sequence}
            \begin{equation}\label{eq:crystalline to TP}
                \E_2^{s,t}=\H^{s-t}_{\crys}(X/W)\Rightarrow\TP_{-s-t}(X),
            \end{equation}
            with differentials $d_r$ of bidegree $(r,1-r)$.
            This is the topological analog of~\eqref{eq:de Rham to HP} and is due to~\cite{bms2}.
        \item[(d)] The Hodge--de
            Rham spectral sequence~\eqref{eq:Hodge to de Rham} is replaced with the {\bf slope spectral sequence}
            \begin{equation}\label{eq:slope SS}
                \E_1^{s,t}=\H^t(X,W\Omega^s_X)\Rightarrow\H^{s+t}_{\cris}(X/W)
            \end{equation}
            of~\cite{illusie-derham-witt}; it has differentials $d_r$ of
            bidegree $(r,1-r)$.
    \end{enumerate}
\end{definition}

Figure~\ref{fig:SS diamond 2} gives the topological analog of Figure~\ref{fig:SS diamond 1}.

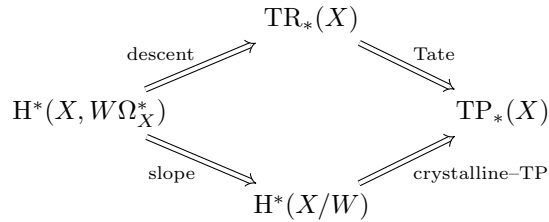
\begin{figure}[H]
  \centering
  \begin{tikzcd}
    &\TR_*(X)\arrow[Rightarrow]{dr}{\text{Tate}}&\\
    \H^*(X,W\Omega^*_X)\arrow[Rightarrow]{ur}{\text{descent}}\arrow[Rightarrow]{dr}[swap]{\text{slope}}&&\TP_*(X)\\
    &\H^*(X/W)\arrow[Rightarrow]{ur}[swap]{\text{crystalline--TP}}&
  \end{tikzcd}
  \caption{Four spectral sequences associated to a smooth proper scheme $X$ over a perfect field $k$ of positive characteristic.}
  \label{fig:SS diamond 2}
\end{figure}

\begin{remark}
    As in the case of Hochschild homology, the Tate spectral sequence is a
    noncommutative invariant, but {\em a priori} the other three spectral sequences are not.
\end{remark}

\begin{remark}\label{rem:differentials are compatible}
    The $W$-modules appearing in Figure \ref{fig:SS diamond 2} are equipped with various extra structures, and the interactions of these structures with the four spectral sequences will be critical in our analysis. Specifically,
    \begin{enumerate}
        \item[{\rm (1)}] the groups $\H^*(W\Omega^*_X)$ and $\TR_*(X)$ are Dieudonn\'{e} modules, and the descent spectral sequence~\eqref{eq:Hesselholt SS} takes place in the abelian category of derived $V$-complete Dieudonn\'e modules. In particular, the differentials on each page commute with $F$ and $V$.
        \item[{\rm (2)}] The groups $\H^*(X/K)=\H^*(X/W)\otimes_WK$ and $\TP_*(X)\otimes_WK$ are $F$-isocrystals (see Definition \ref{def:isocrystal}), and up to certain Tate twists, the crystalline--TP spectral sequence~\eqref{eq:crystalline to TP} is compatible with these $F$-isocrystal structures.
        \item[{\rm (3)}] The differentials in the slope~\eqref{eq:slope SS} and Tate~\eqref{eq:Tate for TP} spectral sequences do \textit{not} commute with $F$ and $V$. Rather, they satisfy the relations in Figure~\ref{fig:Raynaud relations}.
    \end{enumerate}
\end{remark}

\begin{figure}[H]
    \centering
    \begin{tikzcd}[row sep=small]
        &\vdots              & \hspace{.1cm} & \vdots & \hspace{.1cm} & \vdots & \hspace{.1cm}    \\
        \cdots&\TR_{t+1}(\Cscr)            & 0      & \TR_{t+1}(\Cscr) & 0      & \TR_{t+1}(\Cscr)  & \cdots\\
        \cdots&\TR_t(\Cscr)\arrow{urr} & 0      & \TR_t(\Cscr)\arrow{urr}& 0 & \TR_t(\Cscr)  & \cdots\\
        \cdots&\TR_{t-1}(\Cscr)\arrow{urr} & 0      & \TR_{t-1}(\Cscr)\arrow{urr}& 0      & \TR_{t-1}(\Cscr)  & \cdots\\
        &\vdots&&\vdots&&\vdots
    \end{tikzcd}
    \caption{A portion of the $\E_2$-page of the Tate spectral sequence~\eqref{eq:Tate for TP}
        computing $\TP$. The Tate spectral sequence is $2$-periodic in the
        columns and for $\Cscr$ smooth and proper over a perfect field of
        characteristic $p$ it is bounded in the rows
        by~\cite[Corollary~5]{antieau-nikolaus}.}
\label{eq:Tate for TP visual}
\end{figure}

We will need the following proposition.

\begin{proposition}\label{prop:rational degeneration}
    If $X$ is smooth and proper over a perfect field of positive characteristic
    $p$, then the four spectral sequences of Definition~\ref{def:four spectral
    sequences} degenerate rationally.
\end{proposition}

\begin{proof}
    For the slope spectral sequence, this is due to Illusie--Raynaud~\cite{illusie-raynaud}.
    For the crystalline--$\TP$ spectral sequence, it is proved by Elmanto in~\cite{elmanto}
    using an argument of Scholze. The other two cases follow by counting
    dimensions.
\end{proof}

\begin{remark}
    Suppose that $\Cscr$ is a smooth proper dg category over a perfect field
    $k$ and suppose that $\Cscr$ lifts to a smooth proper dg category
    $\widetilde{\Cscr}$ over $W$. Then, $\TP(\Cscr)$ is
    equivalent to $\HP(\widetilde{\Cscr}/W)$ by an unpublished
    argument of Scholze. In this way, one might view $\TP$ as a noncommutative
    version of crystalline cohomology. See also forthcoming work of Petrov and
    Vologodsky.
\end{remark}

\section{Compatibility of the descent and slope spectral
sequences}\label{sec:compatible}

For the remainder of the paper, we use the following notation.

\begin{notation}\label{notation:k}
  Unless otherwise stated, $k$ denotes a perfect field $k$ of characteristic $p>0$. We write $W=W(k)$ for the ring of Witt vectors of $k$ and $K=W[p^{-1}]$ for the field of fractions of $W$. We let $\sigma$ denote the Frobenius on $W$.
\end{notation}

Let $X$ be a smooth proper scheme over $k$. The Hodge--Witt cohomology groups
$\H^*(W\Omega^*_X)$ and the homotopy groups $\TR_*(X)$ are modules over $W$ and are equipped with $\sigma$-linear operators
$F,V$, which satisfy the relations $FV=VF=p$. That is, they are Dieudonn\'{e}
modules in the sense of Definition \ref{def:dieudonne modules etc}. The
differentials
$\H^j(W\Omega^i_X)\rightarrow\H^j(W\Omega^{i+1}_X)$ and
$\TR_i(X)\rightarrow\TR_{i+1}(X)$ do not commute with $F$ and $V$, but instead
satisfy the relations of Figure \ref{fig:Raynaud relations}. Following Illusie--Raynaud \cite{illusie-raynaud}, we formalize these
properties using the {\bf Cartier--Dieudonn\'{e}--Raynaud algebra} relative to
$k$, which is the graded ring
\begin{equation}
    R=R^0\oplus R^1
\end{equation}
generated by $W$ and by operations $F$ and $V$ in degree $0$ and by $d$ in degree
$1$, subject to the relations of Figure~\ref{fig:Raynaud relations}.

\begin{figure}[H]
    \centering
    \begin{align*}
        FV&=VF=p\\
        Fa&=\sigma(a)F\mbox{ for all }a\in W\\
        V\sigma(a)&=aV\mbox{ for all }a\in W\\
        da&=ad\mbox{ for all }a\in W\\
        d^2&=0\\
        FdV&=d
    \end{align*}
    \caption{The relations in the Raynaud ring.}
    \label{fig:Raynaud relations}
\end{figure}

Note that $VF=p=FV$ and $FdV=d$ imply the standard relations $Vd=pdV$ and $dF=pFd$. In particular, $R^0=W_{\sigma}[F,V]$ is the usual Raynaud algebra. Note also that a (left) module over $R^0$ is the same thing as a Dieudonn\'{e} module in the sense of Definition \ref{def:dieudonne modules etc}. We will use the terms interchangeably. A graded left $R$-module is a complex
\[
  M^\bullet=[\cdots\to M^{i-1}\xrightarrow{d} M^i\xrightarrow{d}M^{i+1}\to\cdots]
\]
of $R^0$-modules whose differential $d$ satisfies $FdV=d$. We will refer to
such an object simply as an $R$\textbf{-module}.
\begin{remark}
  A similar structure is studied by Bhatt--Lurie--Mathew in~\cite{blm}. They introduce the notion of a Dieudonn\'e complex. A saturated Dieudonn\'e complex in the sense of~\cite[Definition~2.2.1]{blm} is naturally an $R$-module for the Raynaud algebra relative to $\FF_p$ by~\cite[Proposition~2.2.4]{blm}.
\end{remark}

If $M$ is an $R$-module, we denote by $M[n]$ the graded module with degree
shifted by $n$ so that $M[n]^i=M^{i-n}$. Note that this notation is not consistent with that of
\cite{ekedahl,MR765411}, although we will use the sign conventions from {\em op. cit}.

For each $j$, the differentials in the de Rham--Witt complex make the complex
\begin{equation}\label{eq:Here is an R module}
  \H^{j}(W\Omega^{\bullet}_X)\defeq[0\to\H^j(W\mathscr{O}_X)\xrightarrow{d}\H^j(W\Omega^1_X)\xrightarrow{d}\H^j(W\Omega^2_X)\xrightarrow{d}\cdots]
\end{equation}
into an $R$-module, where $\H^j(W\Omega^i_X)$ is placed in degree $i$ (see~\cite{illusie-derham-witt}).\footnote{In \cite{MR765411} this complex is denoted
by $\R^j\Gamma(X,W\Omega^{\bullet}_X)$ (see page 190). We will avoid this
notation due to its potential for confusion with the $j$-th hypercohomology of
$W\Omega^{\bullet}_X$.} Similarly, by~\cite[Section~6.2]{antieau-nikolaus}, the $S^1$-action on $\TR(X)$ gives rise to a complex
\begin{equation}\label{eq:TR complex}
  \TR_{\bullet}(X)\defeq
  [\cdots\xrightarrow{d}\TR_{t-1}(X)\xrightarrow{d}\TR_t(X)\xrightarrow{d}\TR_{t+1}(X)\xrightarrow{d}\cdots]
\end{equation}
which is an $R$-module, where $\TR_t(X)$ is placed in degree $t$.

The $\E_1$-page of the slope spectral sequence~\eqref{eq:slope SS} is the same
as the $\E_2$-page of the descent spectral sequence~\eqref{eq:Hesselholt SS},
and is depicted in Figure \ref{fig:E1 and E2 pages}.

\begin{figure}[H]
    \centering
\begin{equation*}
    \begin{tikzcd}
      &\vdots&\vdots&\vdots&\\
      \cdots\arrow{r}&\H^{t+3}(W\Omega^{s-1}_X)\arrow{r}&\H^{t+3}(W\Omega^s_X)\arrow{r}
        &\H^{t+3}(W\Omega^{s+1}_X)\arrow{r}&\cdots\\
      \cdots\arrow{r}&\H^{t+2}(W\Omega^{s-1}_X)\arrow{r}&\H^{t+2}(W\Omega^s_X)\arrow{r}&\H^{t+2}(W\Omega^{s+1}_X)\arrow{r}&\cdots\\
      \cdots\arrow{r}&\H^{t+1}(W\Omega^{s-1}_X)\arrow{r}\arrow{uur}
        &\H^{t+1}(W\Omega^s_X)\arrow{r}\arrow{uur}&\H^{t+1}(W\Omega^{s+1}_X)\arrow{r}
        &\cdots\\
      \cdots\arrow{r}&\H^t(W\Omega^{s-1}_X)\arrow{r}\arrow{uur}&\H^t(W\Omega^s_X)\arrow{r}\arrow{uur}&\H^t(W\Omega^{s+1}_X)\arrow{r}&\cdots\\
            &\vdots&\vdots&\vdots&
    \end{tikzcd}
\end{equation*}
    \caption{The $\E_1$-page of the slope spectral sequence~\eqref{eq:slope SS} (with horizontal differentials) and the $\E_2$-page
        of the descent spectral sequence for $\TR$~\eqref{eq:Hesselholt SS} (with diagonal differentials).}
    \label{fig:E1 and E2 pages}
\end{figure}

Let $\F^\star\TR(X)=\R\Gamma(X,\tau_{\geq\star}\TR(\Oscr_X))$ be the filtration
giving rise to the descent spectral sequence.
The differentials in the slope and descent spectral sequences are compatible in the following sense.

\begin{lemma}\label{lem:compatible}
    Let $X$ be a smooth proper variety over a perfect field $k$ of positive
    characteristic. The de Rham--Witt differential
    $\H^t(W\Omega^s_X)\rightarrow\H^t(W\Omega^{s+1}_X)$ is compatible with
    the descent spectral sequence~\eqref{eq:Hesselholt SS}; specifically,
    \begin{enumerate}
        \item[{\rm (i)}] there is a self-map
            $\E_*^{s,t}\rightarrow\E_*^{s+1,t}$ of degree $(1,0)$ of the descent spectral
            sequence which on the $\E_2$-page is the de Rham--Witt
            differential;
        \item[{\rm (ii)}] for each $r\geq 2$, the resulting complexes
            \[
              \E_r^{\bullet,t}=[\cdots\to \E_r^{s-1,t}\to \E_r^{s,t}\to \E_r^{s+1,t}\to\cdots ]
            \]
            are naturally $R$-modules;
        \item[{\rm (iii)}] for each $t\in\ZZ$, the differential $d_r$ in the descent spectral
            sequence is a map of $R$-modules
            $\E_r^{\bullet,t}\rightarrow\E_r^{\bullet,t+r}[r-1]$;
        \item[{\rm (iv)}] the filtration on $\TR_*(X)$ coming from the spectral
            sequence may be interpreted as a filtration by $R$-modules 
            \begin{equation*}
                0=F_{\bullet}^{\dim(X)+1}\subseteq
                F_{\bullet}^{\dim(X)}\subseteq\dots\subseteq
              F_{\bullet}^{1}\subseteq F_{\bullet}^{0}=\TR_{\bullet}(X),
            \end{equation*}
            where
            $$F_i^t=\im\left(\pi_i\F^{i+t}\TR(X)\rightarrow\TR_i(X)\right)$$
            with isomorphisms
            \begin{equation*}
              F_{\bullet}^{t}/F_{\bullet}^{t+1}\cong \E_{\infty}^{\bullet,t}[t].
            \end{equation*}
    \end{enumerate}
\end{lemma}

Note that for each $i$, the filtration on $\TR_i(X)$ induced by the $F_i^t$ is
a shift of the usual filtration coming from the descent spectral sequence.

\begin{proof}
    Consider the sheaf $\TR(\Oscr_X)$ of spectra with $S^1$-action on the
    Zariski site of $X$ and the filtration
    $\F^\star\TR(\Oscr_X)=\tau_{\geq\star}\TR(\Oscr_X)$ arising from the Postnikov
    tower in sheaves of spectra with $S^1$-action. The graded piece
    $\gr^t\TR(\Oscr_X)$ is equivalent to $W\Omega^t_X[t]$. Taking global
    sections, we obtain a filtration
    $\F^\star\TR(X)=\F^\star\R\Gamma(X,\TR(\Oscr_X))$ with graded pieces
    $\gr^t\TR(X)=\R\Gamma(X,W\Omega^t_X)[t]$. The descent spectral sequence
    is by definition the spectral sequence of this filtration (with a
    commonly-used
    reindexing so that it begins with the $\E_2$-page). Now, consider the $S^1$-action on $\TR(\Oscr_X)$. This induces a map
    $\TR(\Oscr_X)\rightarrow\TR(\Oscr_X)[-1]$
    which automatically respects the filtration since it is just the canonical
    Postnikov filtration. In particular, this means that
    $\TR(X)\rightarrow\TR(X)[-1]$ induces a filtered map
    $$\F^\star\TR(X)\rightarrow\F^{\star+1}\TR(X)[-1].$$
    It follows that there is a map of spectral sequences
    associated to the two filtrations. But the spectral sequence coming from
    $\F^{\star+1}\TR(X)[-1]$ is just a re-grading of spectral sequence coming
    from $\F^\star\TR(X)$. In particular, we can view the differential as a
    self-map of the spectral sequence~\eqref{eq:Hesselholt SS}
    of degree $(1,0)$. On the graded pieces of the filtrations, we get maps
    $\R\Gamma(X,W\Omega^t_X)[t]\rightarrow\R\Gamma(X,W\Omega^{t+1}_X)[t].$
    Hesselholt checks in~\cite[Theorem~C]{Hesselholt} that this map is induced
    by the de Rham--Witt differential $W\Omega^t_X\rightarrow W\Omega^{t+1}_X$.
    In particular, on the $\E_2$-page, we see the de Rham--Witt differential.
    This proves part (i).

    We have already remarked that the spectral sequence~\eqref{eq:Hesselholt
    SS} takes places in the abelian category of Dieudonn\'e modules; in other
    words, the differentials in the descent spectral sequence commute with $F$
    and $V$. But, by part (i), the differentials in the descent spectral
    sequence also commute with the de Rham--Witt differential $d$ in the sense that
    for all $t$ and $r\geq 2$ the diagram
    $$\xymatrix{
        \E_r^{s,t}\ar[r]\ar[d]&\E_r^{s+1,t}\ar[d]\\
        \E_r^{s+r-1,t+r}\ar[r]&\E_r^{s+r,t+r}
    }$$ commutes. Now, parts (ii) and (iii) follow by induction where the base
    case is the $R$-module~\eqref{eq:Here is an R module}.

    Finally, to prove part (iv), it is enough to note that the filtered map
    $d\colon\F^\star\TR(X)\rightarrow\F^{\star+1}\TR(X)[-1]$ implies that the
    inclusion $F^t_\bullet\subseteq\TR_\bullet(X)$ is compatible with the
    differential. Since it is also compatible with $F$ and $V$ and since
    $\TR_\bullet(X)$ is an $R$-module, it follows that $F^t_\bullet$ is an
    $R$-submodule. This completes the proof.
\end{proof}

To investigate the fine structure of the Hodge--Witt cohomology groups, Illusie
and Raynaud \cite{illusie-raynaud} introduced a certain subcategory of the
category of $R$-modules, which we briefly review. Forgetting the $F$ and $V$
operations, an $R$-module gives rise to a complex of $W(k)$-modules with
cohomology groups $\H^*(M)$. This complex comes equipped with a
canonical decreasing $(V+dV)$-filtration given by $\mathrm{Fil}^n M^i=V^n M^i+dV^n
M^{i-1}\subseteq M^i$. We say that $M$ is {\bf complete} if each $M^i$ is
complete and separated for the $(V+dV)$-topology. We say that a complete
$R$-module $M$ is {\bf profinite} if $M^i/\mathrm{Fil}^n M^i$ has finite length as a
$W(k)$-module for each $i$ and $n$.
An $R$-module $M$ is {\bf
coherent} if it is bounded (i.e., $M^i=0$ for $|i|$ sufficiently large),
profinite, and $\H^i(M)$ is a finitely generated
$W$-module for all $i$ (see \cite[Th\'{e}or\`{e}me~I.3.8]{illusie-raynaud} and
\cite[D\'efinition~I.3.9]{illusie-raynaud}).

Illusie and Raynaud showed in~\cite[Th\'eor\`eme~II.2.2]{illusie-raynaud} that
the Hodge--Witt complex $\H^j(W\Omega^\bullet_X)$ is coherent for each
$j$.\footnote{Note that this implies that while some $\H^j(W\Omega^i_X)$
might be non-finitely generated as a $W(k)$-module, the terms appearing in the
$\E_2$-page of the slope spectral sequence~\eqref{eq:slope SS} are all finitely generated
$W(k)$-modules.}
In the derived setting, we consider the $R$-module
$\TR_{\bullet}(X)$~\eqref{eq:TR complex}.

\begin{proposition}\label{prop:TR is coherent}
    If $X$ is a smooth proper $k$-scheme, then $\TR_\bullet(X)$ is a coherent $R$-module.
\end{proposition}

\begin{proof}
  As recorded in \cite[page 191]{MR765411}, the category of coherent $R$-modules is closed under kernels, cokernels, and extensions in the category of all $R$-modules. As remarked above, the Hodge--Witt complex $\H^j(W\Omega^\bullet_X)$ is coherent for each $j$. It follows from
  the compatibility of Lemma~\ref{lem:compatible} that the $R$-modules
  $\E_r^{\bullet,t}$ appearing as the rows of the pages in the descent spectral
  sequence for $X$ are coherent for each $r\geq 2$. The descent spectral
  sequence degenerates at some finite stage for degree reasons, so the
  $\E_{\infty}^{\bullet,t}$ are also coherent. But $\TR_{\bullet}(X)$ admits a
  finite filtration by $R$-submodules
  whose successive quotients are isomorphic to the $\E_{\infty}^{\bullet,t}$
  by Lemma~\ref{lem:compatible}.
  The category of coherent $R$-modules is closed under extensions, so it
  follows inductively that each piece of the filtration is coherent. In
  particular, $\TR_{\bullet}(X)$ is coherent, as desired.
\end{proof}

Using results of Illusie-Raynaud and Ekedahl, we obtain the following consequences for the structure of the $R$-module $\TR_{\bullet}(X)$.

\begin{proposition}\label{prop:finiteness for some TR}
    If $X$ is a smooth proper $k$-scheme of dimension $d$ over a perfect field
    of positive characteristic, then for each $i\geq d-2$,
    \begin{enumerate}
        \item[{\rm (1)}] $\TR_i(X)$ is finitely generated over $W$, and
        \item[{\rm (2)}] the differential $d:\TR_{i-1}(X)\to\TR_{i}(X)$ vanishes.
    \end{enumerate}
\end{proposition}
\begin{proof}
   As explained in \cite[Section 3.1]{MR726420}, the domino numbers $T^{i,j}(X)$ of $X$ are zero if $j\leq 1$ or $i\geq d-1$, and therefore $\H^j(W\Omega^i_X)$ is finitely generated over $W$ if $j\leq 1$ or $i=d$ (we review the definition of the domino numbers in Section \ref{sec:domino numbers}). The descent spectral sequence then gives (1).
   Claim (2) then follows from Proposition~\ref{prop:TR is coherent} and Lemma~\ref{lem:another little lemma}.
\end{proof}

\begin{lemma}\label{lem:another little lemma}
  Let $M^\bullet$ be a coherent $R$-module. For each $i$, the following are equivalent.
  \begin{enumerate}
      \item[{\rm (i)}] $M^i$ is finitely generated over $W$.
      \item[{\rm (ii)}] The differentials $M^{i-1}\to M^i$ and $M^i\to M^{i+1}$ vanish.
  \end{enumerate}
\end{lemma}
\begin{proof}
    That (ii) implies (i) follows from the definition of coherence. 
    For the converse, see~\cite[Corollaire II.3.8,
    II.3.9]{illusie-raynaud} for the case of the de Rham--Witt complex
    and~\cite[II.3.1(f)]{illusie-raynaud} for a general coherent $R$-module.
\end{proof}

\section{Derived invariants of varieties}\label{sec:derived invariants from THH}

Let $X$ be a smooth proper variety over $k$ of dimension $d$. By construction, the Dieudonn\'{e} modules $\TR_*(X)$ together with their differentials are derived invariants of $X$. Our goal in the following sections is to relate their structure to that of the crystalline and Hodge--Witt cohomology groups of $X$. Our main tool is the descent spectral sequence~\eqref{eq:Hesselholt SS}, which relates the $\TR_*(X)$ to the Hodge--Witt cohomology groups $\H^*(W\Omega^*_X)$ of $X$.

We begin with the observation that the descent spectral sequence induces natural isomorphisms
\[
    \TR_{-d}(X)\iso\H^d(W\mathscr{O}_X)\quad\text{and}\quad\TR_d(X)\iso\H^0(W\Omega^d_X).
\]
It follows that the Dieudonn\'{e}-modules $\H^d(W\mathscr{O}_X)$ and
$\H^0(W\Omega^d_X)$ are derived invariants. We record a few easy consequences
of this. Recall that if $X$ is Calabi--Yau, then
$\H^d(W\mathscr{O}_X)$ is the Dieudonn\'{e} module associated to the
Artin--Mazur formal group $\Phi^d(X,\mathbb{G}_m)$ of $X$
(see~\cite{artin-mazur}).

\begin{theorem}\label{thm:formal group}
    If $X$ is Calabi--Yau, then the formal group $\Phi^d(X,\mathbb{G}_m)$ is a derived
    invariant, as is the height of $X$.
\end{theorem}

In the case of K3 surfaces over an algebraically closed field, this is well
known, being an easy consequence of the derived invariance of the rational
Mukai crystal $\widetilde{\H}(X/K)$ introduced in section 2.2 of
\cite{MR3429474}.

\begin{remark}
    In fact, the formal group $\Phi^d(X,\Gm)$ is also a twisted derived
    invariant. Indeed, if $\alpha\in\Br(X)$, then
    $\TR_{-d}(X,\alpha)\iso\H^d(W\Oscr_X)$ as well.
\end{remark}

\begin{example}
    Theorem~\ref{thm:formal group} is especially interesting in light of
    Yobuko's result~\cite{yobuko} which says that any finite-height Calabi--Yau
    variety lifts to $W_2(k)$. We see that if $X$ and $Y$ are derived
    equivalent Calabi--Yau varieties and if $X$ has finite height, then so does
    $Y$ and they both lift to $W_2(k)$. This gives some cases in which Question~\ref{question:lieblich} has a positive answer.
\end{example}

The relationship between the remaining Hodge--Witt cohomology groups and the
$\TR_*(X)$ is subtle in general. There are several difficulties here: first,
there may be non-trivial differentials in the descent spectral sequence, which
need to be understood to relate the Hodge--Witt cohomology groups to the
$\E_{\infty}$-page. Second, the $\E_{\infty}$-page only
determines the graded pieces of the induced filtration on $\TR_{\bullet}(X)$.
Unlike in the case of Hochschild homology, this is not enough to determine
$\TR_{\bullet}(X)$ up to isomorphism, as there are non-trivial extensions
in the category of $R$-modules. Finally, as in the case of Hochschild homology,
to compute the $\E_{\infty}$-page from $\TR_{\bullet}(X)$ one needs the
additional data of the filtration on $\TR_{\bullet}(X)$, which is not a derived
invariant.

In the remainder of this section we will attempt to overcome these obstacles in various situations. We begin in Section \ref{sec:first observations} by analyzing the case of surfaces in detail. Here, the situation is sufficiently restricted to allows us to (almost) prove that the entire first page of the slope spectral
sequence is a derived invariant.

We then discuss the groups $\TR_*(X)$ up to isogeny, that is, after inverting $p$. This is the classical theory of \textit{slopes}. A major
simplification occurs here, in that all of the relevant spectral sequences
degenerate after inverting $p$ by Proposition~\ref{prop:rational degeneration}. The resulting analysis follows the well
established patterns involved in extracting derived invariants from Hochschild
homology (see also Remark \ref{rem:isogeny is easy}).

We then study the torsion in the Hodge--Witt cohomology, which contains information lost upon inverting $p$. We focus on the dominoes associated to differentials on the first page of the slope spectral sequence. These encode in an elegant way the infinitely generated
$p$-torsion in the Hodge--Witt cohomology groups, which is present in many
interesting (and geometrically well behaved) examples, such as certain K3 surfaces and abelian varieties. Here we introduce new non-classical derived invariants, defined using the differentials in the Tate spectral sequence~\eqref{eq:Tate for TP}. We then discuss Hodge--Witt numbers, which combine the information from the slopes of isocrystals and the domino numbers. 

We finally apply these results to study the derived invariance of Hodge numbers in
positive characteristic. We note that in order to obtain information on the
Hodge numbers from Hodge--Witt cohomology, it is necessary to remember the infinite $p$-torsion. The information from the slopes of the isocrystals, while perhaps more elementary, does not suffice.



We are able to obtain the strongest results for surfaces and threefolds. For future use we will record a visualization of the information contained in the
slope and descent spectral sequences in these cases. We record the following lemma.

\begin{lemma}\label{lem:E2degeneration}
    Let $X$ be a smooth proper $k$-scheme of dimension $d$. 
    \begin{enumerate}
        \item[(1)] If $d\leq 2$, then the descent spectral sequence for $X$ degenerates at $\E_2$.
        \item[(2)] If $d=3$, then the only possibly nonzero differentials on the $\E_2$ page of the descent spectral sequence for $X$ are as pictured in Figure \ref{fig:threefoldslope}.
    \end{enumerate}
\end{lemma}

\begin{proof}
    Note that, in general, the descent spectral sequence degenerates at $\E_{d+1}$ for degree reasons. This gives the result for $d=1$. If $d=2$ the
    only possible differentials in the $\E_2$-page are
    $\H^0(W\Oscr_X)\rightarrow\H^2(W\Omega^1_X)$ and
    $\H^0(W\Omega^1_X)\rightarrow\H^2(W\Omega^2_X)$. By
    Proposition~\ref{prop:rational degeneration}, the differentials vanish after inverting $p$. Since
    $\H^2(W\Omega^2_X)$ is torsion-free, the latter differential is zero. The
    former is zero by functoriality and the fact that
    $\H^0(W\Oscr_X)\iso\TR_0(k)$. By the same reasoning, when $d=3$ we conclude that the differentials $\H^0(W\mathscr{O}_X)\to\H^2(W\Omega^1_X)$ and $\H^1(W\Omega^2_X)\to\H^3(W\Omega^3_X)$ vanish.
\end{proof}

Suppose now that $X$ is a smooth proper surface over $k$. The first page of the slope
spectral sequence for $X$ is given in Figure~\ref{fig:surfaceslope}.

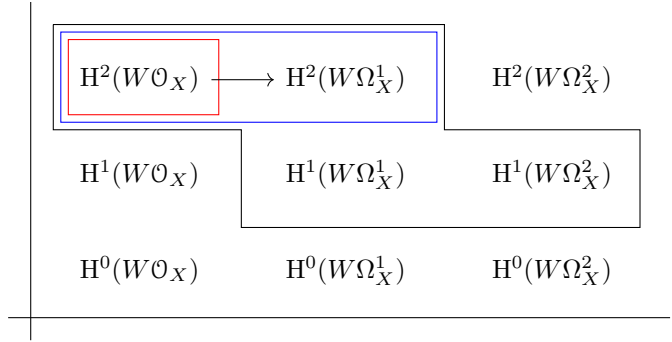
\begin{figure}[H]
    \centering
    \begin{equation*}
    \tikz[
    overlay]{
        \draw[draw=red] (0,1.8) rectangle (2,.8);
        \draw[draw=blue] (-.1,1.9) rectangle (4.9,.7);
        \draw[draw=black] (-.2,2)--(-.2,.6)--(2.3,.6)--(2.3,-.7)--(7.6,-.7)--(7.6,.6)--(5,.6)--(5,2)--(-.2,2);
        \draw[draw=black] (-.5,2.3)--(-.5,-2.2);
        \draw[draw=black] (-.8,-1.9)--(8,-1.9);
    }
        \begin{tikzcd}
          \H^2(W\mathscr{O}_X)\arrow{r}&\H^2(W\Omega^1_X)&\H^2(W\Omega^2_X)    \\
          \H^1(W\mathscr{O}_X) &\H^1(W\Omega^1_X)&\H^1(W\Omega^2_X) \\
         \H^0(W\mathscr{O}_X)  &\H^0(W\Omega^1_X)&\H^0(W\Omega^2_X)
        \end{tikzcd}
    \end{equation*}
    \caption{The $\E_1$-page of the slope spectral sequence and the $\E_2$-page
        of the descent spectral sequence for $\TR$ of a smooth proper surface
        $X$ over $k$.
        The horizontal arrow is the only possibly non-zero differential on
        the first page of the slope spectral sequence. All differentials on
        all pages of the descent spectral sequence vanish (see
        Lemma~\ref{lem:E2degeneration}).
        The red box (the small $1\times 1$ box in the upper left) indicates the source of the only possibly non-zero domino, the blue
        box (the $1\times 2$ box) indicates the possible locations of nilpotent
        torsion, and the black box (the stair-step shaped box)
        indicates the possible locations of semi-simple torsion (see Remark \ref{rem:Ekedahls devissage}).}
    \label{fig:surfaceslope}
\end{figure}

If $X$ is a smooth proper threefold over $k$, then the first page of the slope spectral
sequence is given in Figure~\ref{fig:threefoldslope}.

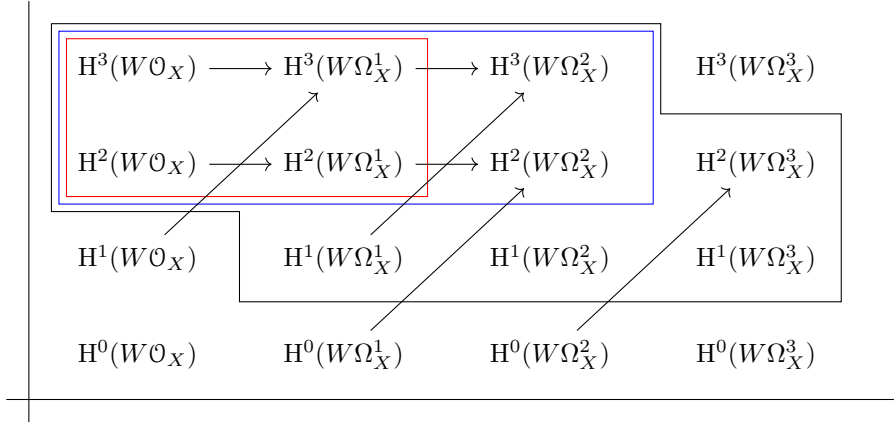
\begin{figure}[H]
    \centering
\begin{equation*}
\tikz[
overlay]{
    \draw[draw=red] (0,2.3) rectangle (4.8,.2);
    \draw[draw=blue] (-.1,2.4) rectangle (7.8,.1);
    \draw[draw=black] (-.2,2.5)--(-.2,0)--(2.3,0)--(2.3,-1.2)--(10.3,-1.2)--(10.3,1.3)--(7.9,1.3)--(7.9,2.5)--(-.2,2.5);
    \draw[draw=black] (-.5,2.8)--(-.5,-2.8);
    \draw[draw=black] (-.8,-2.5)--(11,-2.5);
}
    \begin{tikzcd}
      \H^3(W\mathscr{O}_X)\arrow{r}&\H^3(W\Omega^1_X)\arrow{r} &\H^3(W\Omega^2_X)                      &\H^3(W\Omega^3_X)\\
      \H^2(W\mathscr{O}_X)\arrow{r}&\H^2(W\Omega^1_X)\arrow{r}&\H^2(W\Omega^2_X)                      &\H^2(W\Omega^3_X)\\
      \H^1(W\mathscr{O}_X)\arrow{uur} &\H^1(W\Omega^1_X)\arrow{uur}&\H^1(W\Omega^2_X) &\H^1(W\Omega^3_X)\\
      \H^0(W\mathscr{O}_X)  &\H^0(W\Omega^1_X)\arrow{uur}&\H^0(W\Omega^2_X)\arrow{uur}&\H^0(W\Omega^3_X)
    \end{tikzcd}
\end{equation*}
    \caption{The $\E_1$-page of the slope spectral sequence and the $\E_2$-page
        of the descent spectral sequence for $\TR$ of a smooth proper threefold
        $X$ over $k$.
        The horizontal arrows are all possibly non-zero differentials on
        the first page of the slope spectral sequence. The diagonal arrows are
        all possibly non-zero differentials on the second page of the
        descent spectral sequence (see
        Lemma~\ref{lem:E2degeneration}). The descent spectral sequence degenerates at
        $\E_3$. The red box (the $2\times 2$ box) indicates the sources of the
        possibly non-zero dominoes, the blue box (the $2\times 3$ box) indicates the possible
        locations of nilpotent torsion, and the black box (the stair-step
        shaped box) indicates the
    possible locations of semi-simple torsion (see Remark \ref{rem:Ekedahls devissage}).}
    \label{fig:threefoldslope}
\end{figure}

\subsection{Derived invariants of surfaces}\label{sec:first observations}

Suppose that $X$ is a surface. By Proposition \ref{prop:finiteness for some TR}, the
$R$-module $\TR_{\bullet}(X)$ has only one possibly nonzero differential, and
so looks like
\[
  \TR_{-2}(X)\xrightarrow{d}\TR_{-1}(X)\xrightarrow{0}\TR_0(X)\xrightarrow{0}\TR_1(X)\xrightarrow{0}\TR_2(X).
\]
 Furthermore, $\TR_i(X)$ is finitely generated for $i\geq 0$. By
 Lemma~\ref{lem:E2degeneration}, the descent spectral sequence degenerates at
 $\E_2$, so by Lemma~\ref{lem:compatible} we have a filtration
\[
  0=F^3_{\bullet}\subset F^2_{\bullet}\subset F^1_{\bullet}\subset F^0_{\bullet}=\TR_{\bullet}(X)
\]
by coherent sub-$R$-modules such that $F^i_{\bullet}/F^{i+1}_{\bullet}\cong \H^i(W\Omega^{\bullet}_X)[i]$. This filtration yields short exact sequences
\[
  0\to F^2_{\bullet}\to \TR_{\bullet}(X)\to \frac{\TR_{\bullet}}{F^2_{\bullet}}\to 0\hspace{1cm}\mbox{and}\hspace{1cm}0\to \frac{F^1_{\bullet}}{F^2_{\bullet}}\to \frac{\TR_{\bullet}(X)}{F^2_{\bullet}}\to \frac{\TR_{\bullet}(X)}{F^1_{\bullet}(X)}\to 0
\]
and hence we have commuting diagrams
\begin{equation}\label{eq:big ol diagram}
  \begin{tikzcd}
    &0\arrow{d}&0\arrow{d}&&\\
    \H^2(W\mathscr{O}_X)\isor{d}{}\arrow{r}{d}&\H^2(W\Omega^1_X)\arrow{r}{0}\arrow{d}&\H^2(W\Omega^2_X)\arrow{d}&&\\
    \TR_{-2}(X)\arrow{r}{d}&\TR_{-1}(X)\arrow{r}{0}\arrow{d}&\TR_0(X)\arrow{r}{0}\arrow{d}&\TR_1(X)\arrow{r}{0}\isor{d}{}&\TR_2(X)\isor{d}{}\\
    &\H^1(W\mathscr{O}_X)\arrow{r}{0}\arrow{d}&(F^0_\bullet/F^1_\bullet)_0\arrow{r}{0}\arrow{d}&\TR_1(X)\arrow{r}{0}&\TR_2(X)\\
    &0&0&&
  \end{tikzcd}
\end{equation}
and
\begin{equation}\label{eq:another big ol diagram}
  \begin{tikzcd}
    &0\arrow{d}&0\arrow{d}&\\
    \H^1(W\mathscr{O}_X)\isor{d}{}\arrow{r}{0}&\H^1(W\Omega^1_X)\arrow{r}{0}\arrow{d}&\H^1(W\Omega^2_X)\arrow{d}&\\
    \H^1(W\mathscr{O}_X)\arrow{r}{0}&(F^0_\bullet/F^1_\bullet)_0\arrow{r}{0}\arrow{d}&\TR_1(X)\arrow{r}{0}\arrow{d}&\TR_2(X)\isor{d}{}\\
    &\H^0(W\mathscr{O}_X)\arrow{r}{0}\arrow{d}&\H^0(W\Omega^1_X)\arrow{r}{0}\arrow{d}&\H^0(W\Omega^2_X),\\
    &0&0&&
  \end{tikzcd}
\end{equation}
where the columns are exact sequences of $R$-modules and the rows are $R$-modules.

\begin{theorem}\label{thm:everything is derived invariant for surfaces}
    Let $X$ and $Y$ be smooth proper surfaces over $k$. If $X$ and $Y$ are FM-equivalent, then for
    \[
      (i,j)\in\left\{(0,0),(1,0),(2,0),(1,1),(2,1),(0,2),(2,2)\right\}
    \]
    there exists an isomorphism
    \[
      \H^j(W\Omega^i_X)\cong\H^j(W\Omega^i_Y)
    \]
    of Dieudonn\'e-modules. Furthermore, there exist isomorphisms
    \[
      \H^1(W\mathscr{O}_X)\otimes K\cong\H^1(W\mathscr{O}_Y)\otimes K\hspace{.8cm}\H^2(W\Omega^1_X)\otimes K\cong\H^2(W\Omega^1_Y)\otimes K\hspace{.8cm}\H^2(W\Omega^1_X)[p^{\infty}]\cong\H^2(W\Omega^1_Y)[p^{\infty}]
    \]
    of $R^0$-modules and a commutative diagram
    \[
      \begin{tikzcd}
        \H^2(W\mathscr{O}_X)\arrow{r}{d}\isor{d}{}&\H^2(W\Omega^1_X)[p^{\infty}]\isor{d}{}\\
        \H^2(W\mathscr{O}_Y)\arrow{r}{d}&\H^2(W\Omega^1_Y)[p^{\infty}].
      \end{tikzcd}
    \]
\end{theorem}
\begin{proof}
    By diagram~\eqref{eq:another big ol diagram}, we have an isomorphism $\TR_2(X)\xrightarrow{\sim}\H^0(W\Omega^2_X)$, and by diagram~\eqref{eq:big ol diagram}, we have an isomorphism $\H^2(W\mathscr{O}_X)\xrightarrow{\sim}\TR_{-2}(X)$. This produces the desired isomorphisms for $(i,j)=(2,0),(0,2)$. By Corollary \ref{cor:Popa Schnell consequence}, the isogeny class of $\H^0(W\Omega^1_X)$ is a derived invariant. But $\H^0(W\Omega^1_X)$ is torsion free of slope zero, so in fact $\H^0(W\Omega^1_X)\cong\H^0(W\Omega^1_Y)$. We consider the short exact sequence
    \[
      0\to\H^1(W\Omega^2_X)\to\TR_1(X)\to\H^0(W\Omega^1_X)\to 0
    \]
    whose terms are finitely generated $R^0$-modules,
    and similarly for $Y$. As $\H^0(W\Omega^1_X)$ is torsion free of slope zero, this sequence splits, and we conclude that $\H^1(W\Omega^2_X)\cong\H^1(W\Omega^2_Y)$. We conclude the result for $(i,j)=(0,0),(1,1),(2,2)$ from the derived invariance of $\TR_0(X)$. Consider next the short exact sequence
    \[
      0\to\H^2(W\Omega^1_X)\to\TR_{-1}(X)\to\H^1(W\mathscr{O}_X)\to 0
    \]
    coming from diagram~\eqref{eq:big ol diagram}. As $\H^1(W\mathscr{O}_X)$ is torsion free, the map $\H^2(W\Omega^1_X)\to\TR_{-1}(X)$ induces an isomorphism on torsion.
\end{proof}
\begin{remark}
    Theorem \ref{thm:everything is derived invariant for surfaces} almost shows that the entire first page of the slope spectral sequence is derived invariant.
\end{remark}

As a corollary, we recover the following result. 

\begin{corollary}[{\cite[Corollary 3.4.3]{bragg-derived}}]
    Suppose that $X$ and $Y$ are FM-equivalent K3 surfaces over $k$. If $X$ is supersingular, then so is $Y$, and $\sigma_0(X)=\sigma_0(Y)$.
\end{corollary}
\begin{proof}
    The image of the differential
    \[
      d:\H^2(W\mathscr{O}_X)\to\H^2(W\Omega^1_X)
    \]
    is $p$-torsion. Moreover, the Artin invariant of $X$ is equal to the dimension of the $k$-vector space $\ker d$. This follows for instance from the descriptions given in \cite[Section II.7.2]{illusie-derham-witt}. The result follows from Theorem \ref{thm:everything is derived invariant for surfaces}.
\end{proof}

We also find another proof of the following result of Tirabassi \cite{MR3782696} on derived equivalences of Enriques surfaces.
Recall from~\cite{bombieri-mumford-3} that an Enriques surface $X$ in characteristic $2$ is either classical,
singular, or supersingular, depending on whether the Picard scheme $\Pic_{X/k}$
is $\ZZ/2$, $\mu_2$, or $\alpha_2$, respectively. We call this the type of the Enriques surface.

\begin{corollary}[Tirabassi]
  If $X$ is an Enriques surface over an algebraically closed field of
  characteristic 2, then the type of $X$ is a derived invariant.
\end{corollary}

\begin{proof}
    The $\E_1$ pages of the slope spectral sequences of Enriques surfaces in
    characteristic $2$ are recorded in \cite[Proposition
    II.7.3.6]{illusie-derham-witt}. In particular, we see by Theorem
    \ref{thm:everything is derived invariant for surfaces} that in this case
    the first page of the slope spectral sequence is a derived invariant, and
    this is more than enough to recover the type of $X$.
\end{proof}

\subsection{Slopes and isogeny invariants}\label{sec:slopes and isogeny invariants}

In this section we investigate derived invariants arising from topological Hochschild homology after inverting $p$.
\begin{definition}\label{def:isocrystal}
  An {\bf $F$-isocrystal} is a finite dimensional $K$-vector space $V$ equipped with a $\sigma$-linear map $\Phi\colon V\to V$. A morphism of $F$-isocrystals is a map of vector spaces commuting with the respective semilinear maps. 
\end{definition}
By fundamental results of Dieudonn\'{e} and Manin, it is known that when $k$ is algebraically closed, the category of $F$-isocrystals is abelian semisimple,
and its simple objects are in bijection with rational numbers $\lambda\in\QQ$
(see for example~\cite{demazure}).
The {\bf slopes} of an $F$-isocrystal $(M,\Phi)$ are the collection (with
multiplicities) of the rational numbers appearing in the decomposition of
$M\otimes_KK^{\un}$ into simple objects. Given a subset $S\subset\QQ$, we write
$M_S$ for the sub-isocrystal of $M$ whose slopes are those in the subset $S$.

For a smooth proper $k$-scheme $X$, we let
$\R\Gamma(X/W)=\R\Gamma_\crys(X/W)$ denote the
crystalline cohomology of $X$ over $W$ and we let
$\R\Gamma(X/K)=\R\Gamma(X/W)\otimes_WK$.
Each rational crystalline cohomology group $\H^i(X/K)$ of $X$ comes with an endomorphism
$\Phi$ induced by the absolute Frobenius of $X$, and the pair
$(\H^i(X/K),\Phi)$ is an $F$-isocrystal. Given a rational number $\lambda$, we define the {\bf slope number} of the $i$-th crystalline cohomology of $X$ by
\begin{equation}
    h^{i}_{\cris,\lambda}(X)=\dim_K\H^i(X/K)_{[\lambda]}.
\end{equation}
There are two facts which allow us to get
some control on the slope numbers. The first is that Poincar\'e duality
implies the existence of a perfect pairing
$$\H^i(X/K)\otimes_K\H^{2d-i}(X/K)\rightarrow\H^{2d}(X/K)\iso K(-d)$$ of
isocrystals, where $K(-d)$ is the $1$-dimensional isocrystal of slope
$d$. Thus, for each $\lambda\in\QQ$ we have a perfect pairing
$$\H^i(X/K)_{[\lambda]}\otimes_K\H^{2d-i}(X/K)_{[d-\lambda]}\rightarrow K(-d).$$
This implies
\begin{equation}\label{eq:constraint1}
  h^i_{\cris,\lambda}=h^{2d-i}_{\cris,d-\lambda}.
\end{equation}
Suppose that $X$ is projective. The hard Lefschetz theorem in crystalline cohomology
(see~\cite{katz-messing}) implies
that if $u=c_1(L)\in\H^2(X/K)$ is the rational crystalline Chern class of an
ample line bundle, then cupping with powers of $u$ gives isomorphisms
$$u^i\colon\H^{d-i}(X/K)\iso\H^{d+i}(X/K).$$ Since $u$ generates a
$1$-dimensional subspace of $\H^2(X/K)$ which is closed under Frobenius and has
pure slope $1$, it follows that $u^i$ induces isomorphisms
$$u^i\colon\H^{d-i}(X/K)_{[\lambda]}\iso\H^{d+i}(X/K)_{[i+\lambda]}(i).$$
This implies
\begin{equation*}
    h^{d-i}_{\cris,\lambda}=h^{d+i}_{\cris,i+\lambda},
\end{equation*}
or equivalently
\begin{equation}\label{eq:constraint2}
    h^{i}_{\cris,\lambda}=h^{i}_{\cris,i-\lambda}.
\end{equation}
In fact, by \cite[Corollary 2.2.4]{MR2881317} the relation~\eqref{eq:constraint2} still holds only under the assumption that $X$ is smooth and proper.

The Frobenius endomorphism on rational crystalline cohomology 
comes from an endomorphism of complexes $\R\Gamma(X/W)$. In particular, there
is a Frobenius-fixed $W$-lattice
$\H^i(X/W)/\tors$ inside $\H^i(X/K)$. This implies that the slopes $\lambda$
appearing in crystalline cohomology are all non-negative.
Equation~\eqref{eq:constraint1} implies that the slopes are additionally
bounded above by $d$ and finally~\eqref{eq:constraint2} implies that
the slopes of $\H^i(X/K)$ are bounded above by $i$.

The Hodge--Witt cohomology groups $\H^j(W\Omega^i_X)$ also come with a
$\sigma$-linear operator, denoted by $F$. Again, the pair $(\H^j(W\Omega_X^i)\otimes K,F)$ is an $F$-isocrystal, and given a rational number $\lambda\geq 0$ we write
\begin{equation}
    h^{i,j}_{\dRW,\lambda}(X)=\dim_K\H^j(X,W\Omega^i_X)_{[\lambda]}.
\end{equation}
By \cite[Corollaire II.3.5]{illusie-derham-witt}, we have a canonical isomorphism
\begin{equation}\label{eq:slopes of crystalline coho}
    (\H^{j-i}(W\Omega^i_X),p^iF)\iso\H^{j}(X/K)_{[i,i+1)}
\end{equation}
of $F$-isocrystals. In particular, $h^{i,j}_{\dRW,\lambda}$ is non-zero only if
$\lambda\in[0,1)$ in which case
$$h^{i,j}_{\dRW,\lambda}=h^{i+j}_{\cris,i+\lambda}.$$
It follows from~\eqref{eq:constraint2} that
\begin{equation}\label{eq:constraint3}
    h^{i,j}_{\dRW,\lambda}=h^{i+j}_{\cris,i+\lambda}=h^{d-(d-i-j)}_{\cris,i+\lambda}=h^{2d-i-j}_{\cris,d-j+\lambda}=h^{d-j,d-i}_{\dRW,\lambda},
\end{equation}
which we will use below. This last equality is a crystalline analog
of Hodge symmetry, which we do not have access to in de Rham cohomology.

We now discuss these invariants under FM-equivalence. We record the following
result, which is entirely analogous to Theorem \ref{thm:HH and hodge numbers}
(note however that we need no restrictions on $p$).

\begin{theorem}\label{thm:Newton polygons general}
    If $X$ and $Y$ are FM-equivalent smooth proper $k$-schemes, then there are isomorphisms
    \[
      \bigoplus_{j}\H^{j-i}(W\Omega^j_X)\otimes_W K \cong
      \bigoplus_{j}\H^{j-i}(W\Omega^j_Y)\otimes_W K
    \]
    of $F$-isocrystals. In particular, we have
    \[
      \sum_j h^{j,j-i}_{\dRW,\lambda}(X)=\sum_j h^{j,j-i}_{\dRW,\lambda}(Y)
    \]
    for each $i$ and $\lambda$.
\end{theorem}

\begin{proof}
    By Proposition \ref{prop:rational degeneration}, the descent spectral
    sequence~\eqref{eq:Hesselholt SS} degenerates after tensoring with $K$. We
    therefore obtain a (non-canonical) decomposition
    \begin{equation}
        \TR_i(X)\otimes_W K\cong \bigoplus_{j}\H^{j-i}(W\Omega^j_X)\otimes_W K
    \end{equation}
    of $F$-isocrystals, and similarly for $Y$. The $F$-crystals $\TR_i$ are derived
    invariants, so we get the claimed isomorphism, and the equality of slope numbers
    follows.
\end{proof}

We now define slope numbers for $\TR$ by letting
\[
  h^{\TR}_{n,\lambda}=\dim_K(\TR_n(X)\otimes_WK)_{[\lambda]}.
\]
The proof of Theorem~\ref{thm:Newton polygons general} expressed the fact that
\[
  h^{\TR}_{n,\lambda}=\sum_{i-j=n}h^{i,j}_{\dRW,\lambda}.
\]

We recall the following result of Popa--Schnell (extended to positive
characteristic in Theorem A.1 of \cite{honigs-3} by Achter, Casalaina-Martin,
Honigs, and Vial). 

\begin{theorem}\label{thm:Popa-Schnell}
    If $X$ and $Y$ are FM-equivalent smooth proper varieties over an arbitrary field $k$, then $(\Pic^0_X)_{\red}$ is isogenous to $(\Pic^0_Y)_{\red}$.
\end{theorem}

\begin{remark}
    The theorem is stated in~\cite{honigs-3} only for smooth projective
    varieties. But, the only place projectivity is used in the proof is to
    guarantee the existence of a FM-equivalence, which we assume to exist.
\end{remark}

Equivalently, the isogeny class of the Albanese variety is a derived invariant. This has the following immediate consequence.

\begin{corollary}\label{cor:Popa Schnell consequence}
    If $X$ and $Y$ are FM-equivalent smooth proper varieties over our perfect
    field $k$ of positive characteristic, then 
    the $F$-crystals $\H^1(X/W)$ and $\H^1(Y/W)$ are isogenous and there
    are equalities $h^1_{\crys,\lambda}(X)=h^1_{\crys,\lambda}(Y)$ of slope numbers
    for all $\lambda\in[0,1]$.
\end{corollary}

\begin{remark}\label{rem:Popa Schnell is not enough} 
  In any characteristic, the tangent space to $\Pic^0_X$ at the origin is
  naturally identified with $\H^1(X,\mathscr{O}_X)$. If the characteristic of
  $k$ is zero, then $\Pic^0_X$ is automatically reduced, so Theorem
  \ref{thm:Popa-Schnell} implies that the Hodge number $h^{0,1}$ is a derived
  invariant in characteristic $0$. Similarly, in characteristic $0$, the Hodge
  number $h^{1,0}$ is
  determined by the dimension of the Albanese of $X$, and hence is also a
  derived invariant. In positive characteristic, the isogeny class of the
  Albanese does not in general determine the Hodge numbers $h^{0,1}$ and
  $h^{1,0}$, and therefore Theorem \ref{thm:Popa-Schnell} does not imply the
  invariance of $h^{0,1}$ or $h^{1,0}$ under derived equivalence.
\end{remark}

In dimension $\leq 3$, combining Theorem \ref{thm:Newton polygons general} and
Corollary \ref{cor:Popa Schnell consequence} with the constraints~\eqref{eq:constraint1} and~\eqref{eq:constraint2} give us complete control of
the isocrystals $\H^i(X/K)$ and $\H^j(W\Omega^i_X)\otimes_W K$.

\begin{theorem}\label{thm:Newton polygons}
  Suppose that $X$ and $Y$ are FM-equivalent smooth proper $k$-schemes of dimension $\leq 3$. For each $i,j$, there exist isomorphisms
  \[
    \H^j(W\Omega^i_X)\otimes_W K\cong\H^j(W\Omega^i_Y)\otimes_W K\hspace{1cm}\mbox{and}\hspace{1cm}\H^i(X/K)\cong \H^i(Y/K)
  \]
  of $F$-isocrystals. In particular, we have
  \[
    h^{i,j}_{\dRW,\lambda}(X)=h^{i,j}_{\dRW,\lambda}(Y)\hspace{1cm}\mbox{and}\hspace{1cm}h^{i}_{\cris,\lambda}(X)=h^i_{\cris,\lambda}(Y)
  \]
  for all $i,j,\lambda$.
\end{theorem}

\begin{proof}
    We prove that the derived invariance of the $F$-isocrystal $\H^1(X/K)$
    given by Corollary~\ref{cor:Popa Schnell consequence} and the derived invariance of $\TR$ is enough to
    get the derived invariance of each $h^{i,j}_{\dRW,\lambda}$. This is enough
    to prove the result for Hodge--Witt cohomology and the statement for
    crystalline cohomology follows from the degeneration of the slope spectral
    sequence. We fix $\lambda\in[0,1)$. We know to begin that
    $h^{i,j}_{\dRW,\lambda}$ is a derived invariant for
    $$(i,j)\in\{(0,0),(3,3),(0,1),(1,0),(3,2),(2,3),(0,3),(3,0)\}$$ from the
    derived invariance of
    $\H^0(X/K)$, $\H^6(X/K)$, $\H^1(X/K)$, $\H^5(X/K)$
    (by Poincar\'e duality), $\TR_{-3}(X)$, and $\TR_3(X)$,
    respectively.
    Now,
    \begin{align*}
        h^{\TR}_{2,\lambda}&=h^{3,1}_{\dRW,\lambda}+h^{2,0}_{\dRW,\lambda},\\
        h^{\TR}_{1,\lambda}&=h^{3,2}_{\dRW,\lambda}+h^{2,1}_{\dRW,\lambda}+h^{1,0}_{\dRW,\lambda},\\
        h^{\TR}_{0,\lambda}&=h^{3,3}_{\dRW,\lambda}+h^{2,2}_{\dRW,\lambda}+h^{1,1}_{\dRW,\lambda}+h^{0,0}_{\dRW,\lambda},\\
        h^{\TR}_{-1,\lambda}&=h^{2,3}_{\dRW,\lambda}+h^{1,2}_{\dRW,\lambda}+h^{0,1}_{\dRW,\lambda},\\
        h^{\TR}_{-2,\lambda}&=h^{1,3}_{\dRW,\lambda}+h^{0,2}_{\dRW,\lambda}.
    \end{align*}
    But,
    \begin{align*}
        h^{3,1}_{\dRW,\lambda}&=h^{2,0}_{\dRW,\lambda},\\
        h^{2,2}_{\dRW,\lambda}&=h^{1,1}_{\dRW,\lambda},\\
        h^{1,3}_{\dRW,\lambda}&=h^{0,2}_{\dRW,\lambda}
    \end{align*}
    by~\eqref{eq:constraint3}. When combining this with the derived invariance
    of the Hodge--Witt slope numbers already established above and the
    derived invariance of the $\TR$ slope numbers, we conclude that each de
    Rham--Witt slope number is a derived invariant, as desired.
\end{proof}

We obtain another proof of the following result of Honigs \cite{honigs-3}.
\begin{corollary}
  Suppose that $X$ and $Y$ are smooth proper schemes over a finite field
  $\mathbb{F}_q$ of dimension $\leq 3$. If $X$ and $Y$ are FM-equivalent,
  then $\zeta(X)=\zeta(Y)$.
\end{corollary}

\begin{proof}
    The eigenvalues of Frobenius acting on the $\ell$-adic cohomology of
    $X_{\overline{\FF}_q}$ are determined by the slopes of the crystalline
    cohomology of $X_{\overline{\FF}_q}$. These slopes are derived invariant by
    Theorem~\ref{thm:Newton polygons}.
\end{proof}

Recall that the Betti numbers of $X$ are defined to be $b_n(X)=\dim_K\H^n(X/K)$
for $X$ smooth and proper over $k$. In general, we have only an inequality
$b_n(X)\leq\dim_k\H^n_{\dR}(X/k)$, with equality if and only if $\H^n(X/W)$ and
$\H^{n+1}(X/W)$ are torsion-free.

\begin{corollary}\label{cor:betti numbers}
    If $X$ and $Y$ are FM-equivalent smooth proper $k$-schemes of
    dimension $\leq 3$, then $b_n(X)=b_n(Y)$ for each $n$.
\end{corollary}

\begin{remark}\label{rem:isogeny is easy}
    Given a Fourier--Mukai equivalence $\Phi_P:\D^b(X)\to \D^b(Y)$, the crystalline
    Mukai vector $v(P)$ of $P$ gives rise to a correspondence
    $\H^*(X/K)\to\H^*(Y/K)$. The usual formalism shows that this correspondence
    is an isomorphism of $K$-vector spaces, and the results of Section \ref{sec:slopes and isogeny
    invariants} can be proven by an analysis of the
    K\"unneth components of $v(P)$. Thus, our use of topological constructions,
    while more intrinsic, is not strictly necessary to access the information
    contained in the Hodge--Witt cohomology groups up to isogeny. However,
    the topological constructions appear to be necessary in order to control
    the torsion in the Hodge--Witt cohomology groups. As mentioned above, remembering this information is crucial in order to access the Hodge numbers.
\end{remark}

\subsection{Domino numbers}\label{sec:domino numbers}

To control the infinitely generated $p$-torsion in the Hodge--Witt cohomology
groups, Illusie and Raynaud introduce in~\cite{illusie-raynaud} certain
structures called dominoes and domino numbers.
We review their definition and prove the domino numbers are derived invariants in low
dimensions. If $M$ is an $R$-module, we set
\begin{align*}
    V^{-\infty}Z^iM&=\left\{x\in M^i|dV^n(x)=0\mbox{ for all }n\geq 0\right\}\mbox{ and}\\
    F^{\infty}B^iM&=\left\{x\in M^i|x\in F^nd(M^{i-1})\mbox{ for some }n\geq 0\right\}.
\end{align*}

\begin{definition}
    A coherent $R$-module $M$ is a {\bf domino} if there exists an integer $i$
    such that $M^n=0$ for $n\neq i,i+1$, $V^{-\infty}Z^i=0$, and
    $F^{\infty}B^{i+1}=M^{i+1}$.
\end{definition}

For a further explication of this definition, we refer the reader to
\cite[D\'{e}finition 2.16]{illusie-raynaud} and the surrounding material as
well as~\cite[Section 2.5]{MR726420}. Given any $R$-module $M$, each differential
$M^i\to M^{i+1}$ of $M$ factors as
\begin{equation}\label{eq:dominofactorization}
    \begin{tikzcd}
      M^i\arrow{rrr}{d}\arrow{dr}&&&M^{i+1}\\
      &M^i/V^{-\infty}Z^i\arrow{r}&F^{\infty}B^{i+1}\arrow{ur}&
    \end{tikzcd}
\end{equation}
and, if $M$ is coherent, then the $R$-module 
\[
  \Dom^i(M)=[M^i/V^{-\infty}Z^i\to F^{\infty}B^{i+1}]
\] 
is a domino.
We sometimes refer to this as the domino associated to the differential $M^i\to M^{i+1}$ of $M$.
\begin{definition}\label{def:dimension of domino}
  If $D$ is a domino supported in degrees $i$ and $i+1$, then
  \begin{equation}
      T(D)=\dim_k(D^i/VD^i)
  \end{equation}
  is finite. We refer to it as the \textbf{dimension} of the domino $D$. If $M$
  is a coherent $R$-module, we let $T^i(M)=T(\Dom^i(M))$.
\end{definition}

Illusie and Raynaud define in \cite[Section 1.2.D]{illusie-raynaud} certain
simple one dimensional dominoes $U_{\sigma}$, depending on an integer
$\sigma$. By \cite[Proposition 1.2.15]{illusie-raynaud}, every domino is a
finite iterated extension of the $U_{\sigma}$. We will use the notation
\begin{align}
    \Dom^{i,j}(X)&\defeq\Dom^i(\H^j(W\Omega^{\bullet}_X)),\\
    T^{i,j}(X)   &\defeq T^i(\H^j(W\Omega^{\bullet}_X)).
\end{align}

\begin{example}
  If $X$ is a K3 surface over a perfect field, then the differential
  \begin{equation}\label{eq:a differential}
    d:\H^2(W\mathscr{O}_X)\to\H^2(W\Omega^1_X)
  \end{equation}
  is non-zero if and only if $X$ is supersingular, in which case it is a domino of dimension 1, isomorphic to $U_{\sigma_0}$ where $\sigma_0$ is the Artin invariant of $X$ (see \cite[III.7.2]{illusie-derham-witt}).
\end{example}

We recall that Ekedahl showed in \cite[Theorem IV.3.5]{MR765411} that $\Dom^{i,j}(X)$ and $\Dom^{d-i-2,d-j+2}(X)$ are naturally dual (in a certain sense); as a consequence we have the equality
\begin{equation}\label{eq:duality for dominoes}
    T^{i,j}=T^{d-i-2,d-j+2}
\end{equation}
of Domino numbers \cite[Corollary IV.3.5.1]{MR765411}.

\begin{remark}
   The domino associated to an $R$-module supported in two degrees whose
   differential is zero is zero. Thus, we have $T^{i,j}=0$ if $i\geq d$ or
   $j>d$. By~\eqref{eq:duality for dominoes}, this implies the vanishing of
   various other domino numbers which are not obviously zero. For instance, if
   $X$ is a surface, then we see that the only possible nontrivial domino
   number of $X$ is $T^{0,2}$, the dimension of the domino associated to the
   differential~\eqref{eq:a differential} in the slope spectral sequence.
\end{remark}

\begin{remark}\label{rem:Ekedahls devissage}
    Ekedahl studies in \cite{MR765411} a certain canonical filtration of a
    coherent $R$-module $M$, one piece of which is composed of the dominoes
    defined above (see also \cite{MR726420}). From this filtration Ekedahl
    shows how to partition the torsion of $M$ according to its behavior under
    $V$ and $F$: semisimple torsion, nilpotent torsion, and dominoes. For
    surfaces and threefolds, the possible degrees in which each of these types
    of torsion may appear are indicated in Figures \ref{fig:surfaceslope} and
    \ref{fig:threefoldslope}. We will study only the domino torsion in this
    document, although the semisimple and nilpotent torsion are undoubtedly
    interesting as well.
\end{remark}

Let $X$ be a smooth and proper $k$-scheme.
We now consider the complex $\TR_{\bullet}(X)$, which by Proposition
\ref{prop:TR is coherent} is a coherent $R$-module. We define
\begin{align}
    \Dom^{\cyc}_{i}(X)&\defeq \Dom^i(\TR_{\bullet}(X)),\\
    T^{\cyc}_{i}(X)   &\defeq T^i(\TR_{\bullet}(X)).
\end{align}
By construction, the $T^{\cyc}_i$ are derived invariants of $X$, and we refer to them as the \textbf{derived domino numbers} of $X$. We will use the spectral sequence~\eqref{eq:Hesselholt SS} to relate them to the usual domino numbers $T^{i,j}$ of $X$. We note the following lemma.

\begin{lemma}\label{lem:additivity of domino numbers}
  If $0\to L\to M\to N\to 0$ is an exact sequence of coherent $R$-modules, then for each $i$ we have $T^i(M)=T^i(L)+T^i(N)$.
\end{lemma}
\begin{proof}
    See \cite[Lemma 2.5]{MR3394128}.
\end{proof}

\begin{proposition}\label{prop:top domino numbers is invariant}
    If $X$ is a smooth proper $k$-scheme of dimension $d$, then $T^{0,d}(X)=T^{\cyc}_{-d}(X)$. In particular, $T^{0,d}(X)$ is a derived invariant.
\end{proposition}

\begin{proof}
    We interpret the rows of the pages of the descent spectral sequence for $X$
    as $R$-modules by Lemma~\ref{lem:compatible}. We have an exact sequence
    \[
      \H^{d-2}(W\Omega^{\bullet}_X)[-1]\xrightarrow{d_2}\H^d(W\Omega^{\bullet}_X)\to\E^{\bullet,d}_3\to
      0.
    \]
    The image of $d_2$ is a coherent $R$-submodule
    $\im(d_2)\subseteq\H^d(W\Omega^{\bullet}_X$), and $\Dom^{0}(\im(d_2))=0$ for
    degree reasons. Hence,
    $T^{0,d}(X)=T^0(\H^d(W\Omega^{\bullet}_X))=T^0(\E^{\bullet,d}_3)$. The
    first two terms of $\E^{\bullet,d}_3$ do not see any further differentials,
    and so
    \begin{equation}\label{eq:placeholder}
      T^{0,d}(X)=T^0(\E^{\bullet,d}_3)=T^0(\E^{\bullet,d}_{\infty}).
    \end{equation}
    Consider the filtration $F^i_{\bullet}$ of $\TR_{\bullet}(X)$ from
    Lemma~\ref{lem:compatible}. It follows inductively from Lemma
    \ref{lem:additivity of domino numbers} and the isomorphisms of
    Lemma~\ref{lem:compatible}(iv) that $T^{-d}(\TR_{\bullet}(X)/F^{i}_{\bullet})=0$ for all $i$. Combined with~\eqref{eq:placeholder}, we obtain
    \[
      T^{0,d}(X)=T^0(\E^{\bullet,d}_{\infty})=T^{-d}(F^d_{\bullet})=T^{-d}(\TR_{\bullet}(X))=T^{\cyc}_{-d}(X),
    \]
    as desired.
\end{proof}

\begin{definition}
    Let $X$ be a smooth proper $k$-scheme.
    We say that the descent spectral sequence for $X$ is \textbf{degenerate at
    the level of dominoes} if for each $i,j$ we have
    $T^{i,j}(X)=T^i(\E^{\bullet,j}_r)$ for all $r\geq 2$, where
    $\E_r^{\bullet,j}$ denotes the coherent $R$-module from
    Lemma~\ref{lem:compatible} arising in the descent spectral sequence.
\end{definition}

In low dimensions, this condition is automatic.

\begin{lemma}\label{lem:degeneration for threefolds}
    Let $X$ be a smooth proper $k$-scheme.
    If $X$ has dimension $d\leq 3$, then the descent spectral sequence for $X$ is
    degenerate at the level of dominoes.
\end{lemma}

\begin{proof}
  If $d\leq 2$, the descent spectral sequence is degenerate. Suppose $d=3$. The
  only possibly nonzero dominoes of $X$ are depicted in Figure \ref{fig:threefoldslope}. The result follows immediately from Lemma \ref{lem:additivity of domino numbers}. 
\end{proof}

\begin{proposition}\label{prop:degenerate at the level of dominos}
    Let $X$ be a smooth proper $k$-scheme. If the descent spectral sequence for $X$ is degenerate at the level of
    dominoes, then for each $i$ we have
    \[
      T_i^{\cyc}(X)=\sum_{j\geq 0} T^{i+j,j}(X).
    \]
\end{proposition}

\begin{proof}
    We have
    \[
      T^{i}(F^j_{\bullet}/F^{j+1}_{\bullet})=T^i(\E^{\bullet,j}_{\infty}[j])=T^{i+j}(\E^{\bullet,j}_{\infty})=T^{i+j,j}(X).
    \]
    The result follows from Lemma \ref{lem:additivity of domino numbers}.
\end{proof}

\begin{theorem}\label{thm:domino numbers}
  If $X$ and $Y$ are FM-equivalent smooth proper $k$-schemes of dimension $\leq 3$, then for all $i,j$ we have $T^{i,j}(X)=T^{i,j}(Y)$.
\end{theorem}

\begin{proof}
  If $X$ and $Y$ are surfaces, the result follows from Proposition
  \ref{prop:top domino numbers is invariant}. The only possibly nonzero domino
  numbers of a threefold are $T^{0,2},T^{0,3},T^{1,2}$, and $T^{1,3}$. By
  Proposition \ref{prop:top domino numbers is invariant}
  $T^{0,3}=T^{\cyc}_{-3}$ is a derived invariant. By duality~\eqref{eq:duality
  for dominoes}, $T^{1,2}=T^{0,3}$ is also derived invariant. By Lemma
  \ref{lem:degeneration for threefolds}, the descent spectral sequence of a
  threefold is degenerate at the level of dominoes, so by Proposition
  \ref{prop:degenerate at the level of dominos} we have that
  $T^{\cyc}_{-2}=T^{0,2}+T^{1,3}$ is derived invariant. But, by duality again,
  $T^{0,2}=T^{1,3}$, and hence both terms are themselves derived invariant.
\end{proof}

\subsection{Hodge--Witt numbers}\label{sec:Hodge-Witt numbers}

We recall certain $p$-adic invariants introduced by Ekedahl in Section IV of~\cite{ekedahl}.
We refer the reader also to Crew's article~\cite{MR806843} and
Illusie's article~\cite{MR726420}. Let $X$ be a smooth and proper $k$-scheme. We define the {\bf Hodge--Newton numbers} of $X$ by
\begin{equation}
    m^{i,j}=\sum_{\lambda\in[i,i+1)}(i+1-\lambda)h^{i+j}_{\cris,\lambda}+\sum_{\lambda\in[i-1,i)}(\lambda-i+1)h^{i+j}_{\cris,\lambda}.
\end{equation}
One can show that the $m^{i,j}$ are in fact non-negative integers, and by \cite[Lemma VI.3.1]{ekedahl} they satisfy the relations
\begin{align}
    m^{i,j}&=m^{j,i},\\
    m^{i,j}&=m^{d-i,d-j}.
\end{align}
The {\bf Hodge--Witt numbers} of $X$ are defined as
\begin{equation}
  h^{i,j}_W=m^{i,j}+T^{i,j}- 2T^{i-1,j+1} + T^{i-2,j+2}.
\end{equation}
By \cite[Proposition VI.3.2, VI.3.3]{ekedahl} these satisfy
\begin{equation}
    h^{i,j}_W=h^{d-i,d-j}_W,
\end{equation}
and, if $X$ has dimension $d\leq 3$, one has
\begin{equation}
    h^{i,j}_W=h^{j,i}_W.
\end{equation}
By \cite[Theorem IV.3.2, IV.3.3]{ekedahl}, the Hodge--Witt numbers are related
to the Hodge numbers $h^{i,j}=h^j(X,\Omega^i_X)$ by the inequalities
\begin{equation}\label{eq:inequality for HW and hodge}
    h^{i,j}_W\leq h^{i,j}
\end{equation}
and by {\bf Crew's formula}, which states that
\begin{equation}\label{eq:Crew's formula}
    \sum_j(-1)^jh^{i,j}_W=\sum_{j}(-1)^jh^{i,j}=\chi(\Omega^i_X).
\end{equation}
Finally, we have
\begin{equation}\label{eq:HW numbers and Betti numbers}
  b_n=\sum_{i+j=n}h^{i,j}_W
\end{equation}
for each $n$.

As an immediate consequence of Theorem \ref{thm:Newton polygons} and Theorem
\ref{thm:domino numbers}, we have the following result.

\begin{theorem}\label{thm:hodge-witt numbers}
  If $X$ and $Y$ are FM-equivalent smooth proper $k$-schemes of dimension
  $\leq 3$, then for all $i,j$ we have $h^{i,j}_W(X)=h^{i,j}_W(Y)$.
\end{theorem}

\subsection{Hodge numbers}\label{sec:Hodge numbers}

Let us turn to the question of whether the Hodge numbers $h^{i,j}$ are derived
invariants of smooth proper $k$-schemes. The answer to this is known to be yes
up to dimension $3$ in characteristic $0$ by~\cite{popa-schnell}. For curves
(in any characteristic) it is an
easy consequence of the Hochschild--Kostant--Rosenberg (HKR) theorem. For
surfaces in characteristic $0$, it follows from HKR together with Hodge symmetry
and Serre duality; for threefolds in characteristic $0$ it
follows with the addition input of the theorem of Popa and Schnell (Theorem \ref{thm:Popa-Schnell}).

These arguments fail in several places in positive characteristic. First, the
HKR isomorphism is only known to hold in general if $d\leq p$. Second, Hodge
symmetry fails in general already for surfaces. Finally, in positive
characteristic the isogeny class of the Albanese does not determine the Hodge
numbers $h^{0,1}$ or $h^{1,0}$ (see Remark \ref{rem:Popa Schnell is not
enough}).

In the case of surfaces, we are able to
overcome these difficulties with the additional input of our results on
Hodge--Witt numbers from Section \ref{sec:Hodge-Witt numbers}, which in turn
rely on the results on domino numbers of Section \ref{sec:domino numbers}.
Despite its elementary statement, we do not know a direct proof of Theorem
\ref{thm:hodge numbers for surfaces} avoiding topological Hochschild homology
machinery.

\begin{theorem}\label{thm:hodge numbers for surfaces}
  Suppose that $X$ and $Y$ are smooth proper surfaces over an arbitrary field
  $k$. If $X$ and $Y$ are FM-equivalent, then $h^{i,j}(X)=h^{i,j}(Y)$ for all $i,j$.
\end{theorem}
\begin{proof}
  As described in Theorem \ref{thm:HH and hodge numbers}, the HKR isomorphism
  implies that $h^{2,0},h^{1,1},$ and $h^{0,2}$ are derived invariants, as are
  the sums $h^{0,1}+h^{1,2}$ and $h^{1,0}+h^{2,1}$. If $k$ has characteristic
  0, Serre duality and Hodge symmetry give the result. Suppose that $k$ has
  positive characteristic. We can assume $k$ is perfect since passage to the
  perfection does not change the Hodge numbers. Theorem \ref{thm:hodge-witt
  numbers} and~\eqref{eq:Crew's formula} gives that
  $\chi(\Omega^i_X)$ is a derived invariant for each $i$. It follows that
  $h^{0,1}$ and $h^{1,0}$ are derived invariants. By the HKR isomorphism,
  $h^{1,2}$ and $h^{2,1}$ are derived invariants.
\end{proof}

\begin{remark}
    For the positive characteristic case of the preceding theorem, we may
    instead argue as follows. As observed in the beginning of Section
    \ref{sec:first observations}, the cohomology group $\H^2(W\mathscr{O}_X)$,
    together with its $R^0$-module structure (that is, with its action of $F$
    and $V$), is derived invariant. The length of the $V$-torsion of
    $\H^2(W\mathscr{O}_X)$ is equal to the dimension of the tangent space at
    the origin of $\Pic^0_X$ minus the dimension of the tangent space at the
    origin of $(\Pic^0_X)_{\red}$ (see for instance \cite[Remarque
    II.6.4]{illusie-derham-witt}). By Theorem \ref{thm:Popa-Schnell}, the
    latter is a derived invariant as well. It follows that the dimension of the
    tangent space of $\Pic^0_X$ at the origin, which is equal to the Hodge
    number $h^{0,1}$, is a derived invariant. We then conclude by Serre duality
    and HKR, as before.
\end{remark}

We next consider threefolds in positive characteristic. To ensure that the HKR
spectral sequence degenerates, we might restrict our attention to
characteristic $p\geq 3$.\footnote{The examples of~\cite{antieau-bhatt-mathew}
of varieties with non-degenerate HKR spectral sequence $2p$-dimensional. We do
not know an example of a threefold in characteristic 2 with non-degenerate HKR
spectral sequence.} Even with this restriction, the failure of Hodge symmetry
and of Popa--Schnell to determine $h^{0,1}$ means that we do not have enough
control to prove derived invariance of Hodge numbers of threefolds in general
(see however Theorem \ref{thm:Mazur Ogus and derived equivalence}). We record
the following consequence of Theorem \ref{thm:hodge-witt numbers}.

\begin{theorem}\label{thm:hodge numbers for threefolds}
    Suppose that $X$ and $Y$ are smooth proper schemes of dimension $3$
    over $k$, a perfect field of positive characteristic $p$. If $X$ and
    $Y$ are FM-equivalent, then $\chi(\Omega^i_X)=\chi(\Omega^i_Y)$ for
    each $0\leq i\leq 3$.
\end{theorem}

\begin{proof}
   This follows immediately from Theorem \ref{thm:hodge-witt numbers} and~\eqref{eq:Crew's formula}.
\end{proof}

\begin{remark}
    Suppose that $p\geq 3$, so that the spectral sequence~\eqref{eq:HKR spectral sequence} associated to a threefold $X$ degenerates. The HKR isomorphism then gives that certain sums of Hodge numbers of $X$ are derived invariants, as described in Theorem \ref{thm:HH and hodge numbers}. Combined with Serre duality and the obvious Hodge number $h^{0,0}=1$, we obtain 13 linearly independent relations which are preserved by derived equivalences on the 16 total Hodge numbers.
    
    It is not hard to check that the relations in the conclusion of Theorem \ref{thm:hodge numbers for threefolds} are \textit{not} in the span of these relations. Precisely, the result of Theorem \ref{thm:hodge numbers for threefolds} gives exactly one new linear relation on Hodge numbers that is preserved under derived equivalence.
\end{remark}

\subsection{Mazur--Ogus and Hodge--Witt varieties}

In this section, we prove the derived invariance of certain conditions on
de Rham and Hodge--Witt cohomology. We continue Notation~\ref{notation:k}, so that $k$ is a perfect field of positive characteristic $p$.

Following Joshi~\cite[2.31]{joshi-geography}, we make the following definition.

\begin{definition}
    Let $X$ be a smooth proper $k$-scheme. We say that $X$ is {\bf Mazur--Ogus} if
  \begin{enumerate}
      \item[(a)] the Hodge--de Rham spectral sequence for $X$ degenerates at $\E_1$, and
      \item[(b)] the crystalline cohomology groups of $X$ are torsion free.
  \end{enumerate}
\end{definition}
We view this as a rather mild set of assumptions, which still allow for a lot
of interesting behaviors in the Hodge--Witt cohomology of $X$. For instance, K3
surfaces, abelian varieties, and complete intersections in projective space are
Mazur--Ogus.

\begin{lemma}\label{lem:Mazur Ogus lemma}
  If $X$ is a smooth proper $k$-scheme, then the following conditions are equivalent:
  \begin{enumerate}
      \item[{\rm (1)}] $X$ is Mazur--Ogus;
      \item[{\rm (2)}] $b_n(X)=\sum_{i+j=n}h^{i,j}(X)$ for all $n$;
      \item[{\rm (3)}] $h^{i,j}(X)=h^{i,j}_W(X)$ for all $i,j$.
  \end{enumerate}
\end{lemma}

\begin{proof}
  In general, we have inequalities
  \[
    b_n(X)\leq \dim\H^n_{\dR}(X/k)\leq\sum_{i+j=n}h^{i,j}(X)
  \]
  The first of these is an equality if and only if $\H^n(X/W)$ and
  $\H^{n+1}(X/W)$ are torsion free,
  and the second is an equality if and only if the Hodge--de Rham spectral
  sequence in degree $n$ degenerates at $\E_1$. This shows $(1)\iff (2)$.
  Using~\eqref{eq:inequality for HW and hodge} and~\eqref{eq:HW numbers and
  Betti numbers} we deduce $(2)\iff (3)$.
\end{proof}

\begin{theorem}\label{thm:Mazur Ogus and derived equivalence}
    Let $X$ be a smooth proper $k$-scheme of dimension $\leq 3$. If $X$ is
    Mazur--Ogus and if $Y$ is a smooth proper $k$-scheme such that $\Dscr^b(X)\cong
    \Dscr^b(Y)$, then $Y$ is Mazur--Ogus and we have $h^{i,j}(X)=h^{i,j}(Y)$ for all $i,j$.
\end{theorem}

\begin{proof}
  Using Lemma \ref{lem:Mazur Ogus lemma}, Theorem \ref{thm:hodge-witt numbers}, and~\eqref{eq:inequality for HW and hodge} we have
  \[
    h^{i,j}(X)=h^{i,j}_W(X)=h^{i,j}_W(Y)\leq h^{i,j}(Y)
  \]
  for each $i,j$. Using the assumption that $p\geq 3$, we have by Theorem \ref{thm:HH and hodge numbers} that
  \[
    \sum_{j}h^{j,j-i}(X)=\sum_{j}h^{j,j-i}(Y)
  \]
  for each $i$. We conclude that $h^{i,j}(X)=h^{i,j}(Y)$ for all $i,j$ and
  hence $h^{i,j}_W(Y)=h^{i,j}(Y)$ for all $i$ and $j$. By Lemma \ref{lem:Mazur Ogus lemma} we conclude that $Y$ is
  Mazur--Ogus.
\end{proof}

\begin{definition}
  Following~\cite[Section~IV.4]{illusie-raynaud}, we say that a smooth proper
  $k$-scheme is {\bf Hodge--Witt} if $\H^j(X,W\Omega^i_X)$ is finitely
  generated as a $W$-module for all $i$ and $j$. We say that $X$ is {\bf
  derived Hodge--Witt} if $\TR_i(X)$ is finitely generated as a $W$-module for
  all $i$.
\end{definition}

Hodge--Witt is implied by ordinarity, but is weaker
than it. For example, a K3 surface is Hodge--Witt if and only if it is non-supersingular, whereas it is ordinary if and only if the associated formal group has height $1$ (see the \cite[II.7.2]{illusie-derham-witt}).

\begin{proposition}\label{prop:Hodge--Witt}
    Let $X$ be a smooth proper $k$-scheme. The following are equivalent.
    \begin{enumerate}
        \item[{\rm (1)}] $X$ is Hodge--Witt.
        \item[{\rm (2)}] The slope spectral sequence for $X$ degenerates at $\E_1$.
        \item[{\rm (3)}] $T^{i,j}(X)=0$ for all $i,j$.
    \end{enumerate}
\end{proposition}
\begin{proof}
    By~\cite[IV.4.6.2]{illusie-raynaud}, $X$ is Hodge--Witt if and only if the
    slope spectral sequence for $X$ degenerates at $\E_1$, and so we have
    $(1)\iff(2)$. We have $(1)\iff (3)$ by for instance
    \cite[3.1.4]{MR726420}.
\end{proof}

\begin{theorem}\label{thm:Hodge--Witt and derived Hodge--Witt}
    Let $X$ be a smooth proper $k$-scheme. If $X$ is of dimension $\leq 3$, then $X$ is Hodge--Witt if and only if it is derived Hodge--Witt.
\end{theorem}

\begin{proof}
    By \cite[3.1.4]{MR726420}, we have that for $j=0,1$ the Hodge--Witt
    cohomology groups $\H^j(W\Omega^i_X)$ are finitely generated for each $i$.
    The differentials in the descent spectral sequence have vertical degree 2
    (with our conventions), and hence $X$ is Hodge--Witt if and only if the terms $\E_{\infty}^{i,j}$ appearing on the $\E_{\infty}$-page of the descent spectral sequence are finitely generated for all $i,j$.
    
    Furthermore, each of the $\TR_i(X)$ admits a filtration whose successive quotients are given by terms appearing on the $\E_{\infty}$ page of the descent spectral sequence for $X$. In particular, we see that all of the $\TR_i(X)$ are finitely generated $W$-modules if and only if $\E_{\infty}^{i,j}$ is finitely generated for all $i,j$.
\end{proof}

\begin{corollary}\label{cor:hw}
  Let $X$ and $Y$ be smooth proper threefolds over $k$. If $\Dscr^b(X)\we\Dscr^b(Y)$ and if $X$ is Hodge--Witt, then so is $Y$.
\end{corollary}

In particular, the equivalent conditions recorded in Proposition \ref{prop:Hodge--Witt} are all derived invariants in dimension $d\leq 3$.

\begin{example}
    Joshi shows in~\cite[Corollary~6.2]{joshi-exotic} that if $X$ is an
    $F$-split threefold, then it is Hodge--Witt. Thus, Corollary~\ref{cor:hw}
    applies to $F$-split threefolds. This has recently been extended to
    quasi-$F$-split threefolds by Nakkajima~\cite[Corollary~1.8]{nakkajima} and
    in particular applies then to all finite height Calabi--Yau threefolds
    by~\cite{yobuko}.
\end{example}

\begin{remark}
  Using the equivalent conditions of Proposition \ref{prop:Hodge--Witt}, one can give a less direct proof
  of Theorem \ref{thm:Hodge--Witt and derived Hodge--Witt} using Proposition
  \ref{prop:degenerate at the level of dominos}.
\end{remark}

\section{Twisted K3 surfaces}\label{sec:K3}

In this section we will study some examples with interesting behavior in the
slope and descent spectral sequences. Specifically, we will completely
compute $\TR$ and $\TP$ for twisted K3 surfaces. This will give a different
perspective on the twisted K3 crystals defined and studied in
\cite{bragg-lieblich-twistor}. For the remainder of this section, we let $k$ be an algebraically closed field of positive
characteristic $p$.

While we have only discussed derived invariants of varieties so far, much of our discussion carries over unchanged to twisted varieties. For instance, $\TR$ and $\TP$ are defined for an abstract dg category. Given $\alpha\in\Br(X)$ a Brauer class on a smooth proper variety $X$, we let $\TR_*(X,\alpha)$ and $\TP_*(X,\alpha)$ denote their application to the natural enhancement of the bounded derived category of $\alpha$-twisted coherent sheaves on $X$. There is a descent spectral sequence
\begin{equation}\label{eq:twisted Hesselholt SS}
    \E_2^{s,t}=\H^t(X,W\Omega^s_{X})\Rightarrow\TR_{s-t}(X,\alpha)
\end{equation}
which computes the $\TR_*(X,\alpha)$ in terms of the Hodge--Witt cohomology groups of the underlying variety $X$. We remark that the Hodge--Witt cohomology groups of a K3 surface are determined very explicitly in \cite[II.7.2]{illusie-derham-witt}. We will see that, although the objects appearing on the $\E_2$ page of~\eqref{eq:twisted Hesselholt SS} are the same as those on the $\E_2$ page of the descent spectral sequence for $X$, the differentials may be different.
We let $d_r^\alpha$ denote the differentials in the $\alpha$-twisted descent
spectral sequence.

Let $X$ be a K3 surface over $k$. We begin by computing the topological periodic cyclic homology groups $\TP_*(X)$ of $X$ in terms of crystalline cohomology. Since the crystalline cohomology groups of $X$ are all torsion-free, the crystalline--TP spectral sequence~\eqref{eq:crystalline to TP} degenerates; the corresponding filtration splits noncanonically and gives an isomorphism of $W$-modules
    \[
        \TP_{2i}(X)\iso\H^0(X/W)\oplus\H^2(X/W)\oplus\H^4(X/W)
    \]
for each $i$. In general, the $K$-vector spaces obtained by tensoring $\TP_i$
with $\QQ$ admit a natural Frobenius operator coming from the cyclotomic
Frobenius and are so endowed with a structure of $F$-isocrystal.\footnote{By an argument of Scholze, the filtration from the crystalline--$\TP$ spectral sequence can be canonically split rationally by using Adams operations; see~\cite{elmanto}.} However, after inverting $p$ this isomorphism does not carry the Frobenius on the left hand side to the natural Frobenius operator $\Phi$ on crystalline cohomology. Rather, consider the Mukai crystal
    \[
        \widetilde{\H}(X/W)=\H^0(X/W)(-1)\oplus\H^2(X/W)\oplus\H^4(X/W)(1)
    \]
as introduced in \cite{MR3429474}. The above can then be upgraded to isomorphisms $\TP_{2i}(X)_{\QQ}\cong\widetilde{\H}(X/K)(i+1)$ of $F$-isocrystals for each $i$, where $\widetilde{\H}(X/K)=\widetilde{\H}(X/W)\otimes_WK$. In fact, for any (possibly twisted) surface, the Frobenius is defined integrally\footnote{To
see this, one must use the Nygaard filtration.} on $\TP_{i}(X,\alpha)$ for $i\leq -2$. The filtration on $\TP$ can be split even integrally, and so we obtain an isomorphism
    \begin{equation}\label{eq:tricky twists 2}
      \TP_{2i}(X)\cong\widetilde{\H}(X/W)(i+1)
    \end{equation}
of $F$-crystals for each $i\leq -1$.

\subsection{Finite height}

Let $(X,\alpha)$ be a twisted K3 surface over $k$, and suppose that $X$ has finite
height. By the computation of the Hodge--Witt cohomology groups of $X$ in \cite[Section II.7.2]{illusie-derham-witt}, we see that the slope spectral
sequence for $X$ and the descent spectral sequences for $X$ and
$(X,\alpha)$ are all degenerate for degree reasons. Hence, both $\TR_0(X)$ and $\TR_0(X,\alpha)$ admit a filtration by $R^0$-submodules with graded pieces $\H^0(W\mathscr{O}_X)$, $\H^1(W\Omega^1_X)$, and $\H^2(W\Omega^2_X)$. As these groups are all torsion free, this filtration certainly splits at the level of $W$-modules. In fact, by computing the appropriate $\Ext$ groups, one can show that it even splits at the level of $R^0$-modules. We conclude that there exist (non canonical) isomorphisms
\begin{equation}\label{eq:TR computation finite height}
    \TR_i(X)\cong\TR_i(X,\alpha)\cong\begin{cases}
        \H^2(W\Oscr_X)&\text{if $i=-2$,}\\
        \H^0(W\mathscr{O}_X)\oplus\H^1(W\Omega^1_X)\oplus\H^2(W\Omega^2_X)&\text{if $i=0$,}\\
        \H^0(W\Omega^2_X)&\text{if $i=2$}
    \end{cases}
\end{equation}
of $R^0$-modules, and $\TR_i(X)=\TR_i(X,\alpha)=0$ otherwise.

We next compute $\TP$. Because $\TR$ is concentrated in even degrees, the Tate spectral sequences
for $X$ and $(X,\alpha)$ degenerate and $\TP$ is also concentrated in even degrees.
We have a filtration on $\TP_{2i}(X,\alpha)$ with graded pieces
$\H^0(W\Omega^2_X)$, $\TR_0(X,\alpha)$, and $\H^2(W\Oscr_X)$. Moreover, for
$i\leq -1$, keeping track of the appropriate Tate twists this yields a filtration by $F$-crystals,\footnote{By
\cite[II.7.2a]{illusie-derham-witt} the Hodge--Witt cohomology groups of $X$ are finitely generated and torsion free.} whose graded pieces are
$\H^0(W\Omega^2_X)(i-1)$, $\TR_0(X,\alpha)(i)$, and $\H^2(W\Oscr_X)(i+1)$.

In particular, this determines the $\TP_{2i}(X,\alpha)$ as $F$-isocrystals. To
determine their structure as $F$-crystals, one needs some additional input.
This can be done by comparison with the B-field constructions of
\cite{bragg-lieblich-twistor}, from which one can compute the Hodge polygon of
$\TP_{2i}(X,\alpha)$. Using Katz's Newton--Hodge decomposition~\cite[Theorem
1.6.1]{MR563463}, one then deduces that the filtration on $\TP_{2i}(X,\alpha)$
in fact splits canonically, and so we have a canonical isomorphism
\[
  \TP_{2i}(X,\alpha)=\H^2(W\mathscr{O}_X)(i+1)\oplus\TR_0(X,\alpha)(i)\oplus\H^0(W\Omega^2_X)(i-1)
\]
of $F$-crystals for each $i\leq -1$. In particular, by~\eqref{eq:TR computation finite height}, we see that $\TP_{2i}(X)\cong\TP_{2i}(X,\alpha)$ as $F$-crystals for each $i\leq -1$, although not canonically.

\subsection{Supersingular}

We now consider a twisted K3 surface $(X,\alpha)$ where $X$ is supersingular.
We will examine the descent spectral sequence for $(X,\alpha)$. In particular,
we will see that it is not degenerate.

We begin with a few general facts. For any smooth $k$-scheme $X$, there is
a natural map of \'etale sheaves
\begin{equation}\label{eq:dlog map}
  \mathbb{G}_m\xrightarrow{\dlog} W\Omega^1_X
\end{equation}
given on sections by $f\mapsto\tfrac{d[f]}{[f]}$. There is an induced map on cohomology
\begin{equation}\label{eq:dlog map on cohomology}
  \dlog:\H^2(X,\mathbb{G}_m)\to\H^2(W\Omega^1_X).
\end{equation}

\begin{lemma}\label{lem:differential in descent ss}
    Let $X$ be a smooth proper $k$-scheme and let
    $\alpha\in\H^2(X,\mathbb{G}_m)$.
    \begin{enumerate}
        \item[{\rm (a)}]
        The differential
            \[
            d_2^\alpha:\H^0(W\mathscr{O}_X)\to\H^2(W\Omega^1_X)
            \]
            appearing in the $\E_2$-page of the $\alpha$-twisted descent
            spectral sequence~\eqref{eq:twisted Hesselholt SS} sends the
            canonical generator $1\in W=\H^0(W\mathscr{O}_X)$ to $\dlog(\alpha)$.
        \item[{\rm (b)}]
            If $X$ is a surface, then all other differentials in the twisted
            descent spectral sequence are zero.
    \end{enumerate}
\end{lemma}

\begin{proof}
    Let $\K^\et(X,\alpha)$ denote the $\alpha$-twisted \'etale $K$-theory of
    $X$. There is a descent spectral sequence
    $$\E_2^{s,t}=\H^t_\et(X,\K_s(\Oscr_X))\Rightarrow\K^\et_{s-t}(X,\alpha).$$
    We also have natural isomorphisms $\ZZ\iso\K_0(\Oscr_X)$ and
    $\Gm\iso\K_1(\Oscr_X)$. The $d_2^\alpha$-differential
    $\H^0_\et(X,\ZZ)\rightarrow\H^2_\et(X,\Gm)$ sends $1$ to $\alpha$
    by~\cite[Theorem~8.5]{antieau-cech}. Now, the map
    $\K_1(\Oscr_X)\iso\Gm\rightarrow\TR_1(\Oscr_X)\iso W\Omega^1$ is given by
    the $\dlog$ map; see~\cite[Lemma~4.2.3]{geisser-hesselholt}.\footnote{Note that in Geisser--Hesselholt, the map is given by $-\dlog$, but this
    sign depends on a choice of the HKR isomorphism $\TR_1(\Oscr_X)\iso
    W\Omega^1_X$ which amounts to a choice of an orientation on the circle.}
    Thus, part (a) follows from the compatibility between the descent spectral
    sequences for $\K^\et(X,\alpha)$ and $\TR(X,\alpha)$ using the trace map
    $\K^\et(X,\alpha)\rightarrow\TR(X,\alpha)$ and especially the commutative
    diagram
    $$\xymatrix{
    \H^0_\et(X,\ZZ)\ar[r]^{d_2^\alpha}\ar[d]&\H^2(X,\Gm)\ar[d]^{\dlog}\\
    \H^0(W\Oscr_X)\ar[r]^{d_2^\alpha}&\H^2(W\Omega^1_X).
    }$$ For part (b), note that all of the differentials are torsion but that
    $\H^2(W\Omega^2_X)$ is torsion-free. Thus,
    $d_2^\alpha\colon\H^0(W\Oscr_X)\rightarrow\H^2(W\Omega^1_X)$ is the
    only possible non-zero differential for a surface.
\end{proof}

We conclude from the above that the descent spectral sequence for a twisted
surface $(X,\alpha)$ is degenerate at $\E_2$ if and only if $\dlog(\alpha)=0$,
and is always degenerate at $\E_3$ for degree reasons.

Let us now return to the situation where $X$ is a supersingular K3 surface. We record the following result.

\begin{lemma}\label{lem:SES for ssing K3}
    If $X$ is a supersingular K3 surface over an algebraically closed field $k$
    of positive characteristic $p$, then the sequence
    \[
      0\to\Br(X)\xrightarrow{\dlog}\H^2(W\Omega_X^1)\xrightarrow{1-F}\H^2(W\Omega^1_X)\to 0
    \]
    is exact, where the left arrow is the map on cohomology induced by~\eqref{eq:dlog map}.
\end{lemma}
\begin{proof}
  As $\H^1(W\Omega^1_X)$ is finitely generated, its endomorphism $1-F$ is
  surjective by \cite[Lemme II.5.3]{illusie-derham-witt}. By flat duality,
  $\H^4(X,\mathbb{Z}_p(1))=0$. By
  \cite[Th\'{e}or\`{e}me~II.5.5]{illusie-derham-witt} we therefore obtain a
  short exact sequence
  \[
    0\to\H^3(X,\mathbb{Z}_p(1))\to\H^2(W\Omega^1_X)\xrightarrow{1-F}\H^2(W\Omega^1_X)\to 0
  \]
  where as usual we put
  \[
    \H^3(X,\mathbb{Z}_p(1))\defeq \varprojlim \H^3(X,\mu_{p^n}).
  \]
  As a consequence of flat duality, Artin showed that this inverse system is constant, and hence the natural map
  \[
    \H^3(X,\mathbb{Z}_p(1))\xrightarrow{\sim}\H^3(X,\mu_p)
  \]
  is an isomorphism. Furthermore, the boundary map induced by the K\"ummer sequence gives an isomorphism
  \begin{equation}\label{eq:Br=mu}
    \Br(X)\xrightarrow{\sim}\H^3(X,\mu_p).
  \end{equation}
  For these facts, see the proof of \cite[Theorem 4.3]{MR0371899} on page 559.
  We thus find an isomorphism
  $\Br(X)\xrightarrow{\sim}\H^3(X,\mathbb{Z}_p(1))$. Using the definitions of
  the maps involved, one checks that the resulting map
  $\Br(X)\to\H^2(W\Omega^1_X)$ is the map on cohomology induced
  by~\eqref{eq:dlog map}.
\end{proof}

We remark that the isomorphism~\eqref{eq:Br=mu} implies $\Br(X)$ is
$p$-torsion; in fact, as recorded in \cite{MR0371899}, there is an abstract
isomorphism of groups $\Br(X)\cong k$.

Lemma \ref{lem:SES for ssing K3} implies in particular that $\dlog$ is
injective. The Brauer group of a supersingular K3 surface is non-trivial, so
combined with Lemma \ref{lem:differential in descent ss} we have produced
examples of twisted surfaces with non-degenerate descent spectral sequence. We record the $\E_2$ and $\E_3$ pages in the following figure.

\begin{figure}[H]
    \centering
    \begin{equation*}
        \tikz[
    overlay]{
        \draw[draw=black] (-.2,1.8)--(-.2,-2);
        \draw[draw=black] (-.8,-1.6)--(7,-1.6);
    }
        \begin{tikzcd}[column sep=small, row sep=small]
          \H^2(W\mathscr{O}_X)\arrow{r}{d}&\H^2(W\Omega^1_X)&\H^2(W\Omega^2_X)    \\
          0 &\H^1(W\Omega^1_X)& 0 \\
         \H^0(W\mathscr{O}_X)\arrow{uur}{d_2^{\alpha}}&0&0
        \end{tikzcd}
        \hspace{2cm}
        \tikz[
    overlay]{
        \draw[draw=black] (-.2,1.8)--(-.2,-2);
        \draw[draw=black] (-.8,-1.6)--(7,-1.6);
    }
        \begin{tikzcd}[column sep=small, row sep=small]
          \H^2(W\mathscr{O}_X)\arrow{r}{d}&\frac{\H^2(W\Omega^1_X)}{\dlog(\alpha)}&\H^2(W\Omega^2_X)    \\
          0 &\H^1(W\Omega^1_X)& 0 \\
            \mathrm{ord}(\alpha)W&0&0
        \end{tikzcd}
    \end{equation*}
    \caption{The $\E_2$ and $\E_3=\E_{\infty}$ pages of the descent spectral
    sequence for $(X,\alpha)$. The horizontal arrows are the maps induced by
the differentials appearing on the $\E_1$ page of slope spectral sequence for
$X$. The term $\mathrm{ord}(\alpha)W$ is the kernel of
the non-zero differential, which is generated by
$\mathrm{ord}(\alpha)\in W$, where $\mathrm{ord}(\alpha)=\mathrm{ord}(\dlog(\alpha))$ is the
order of $\alpha$.}
    \label{fig:twisted surface slope}
\end{figure}

As in the finite height case, we therefore have noncanonical isomorphisms
\begin{equation}\label{eq:TR computation}
  \TR_0(X)\cong\TR_0(X,\alpha)\cong\H^0(W\mathscr{O}_X)\oplus\H^1(W\Omega^1_X)\oplus\H^2(W\Omega^2_X),
\end{equation}
of $R^0$-modules, where we use $\mathrm{ord}(\alpha)$ to give an isomorphism between $W\iso\H^0(W\Oscr_X)$ and $\mathrm{ord}(\alpha)W$. We have a commuting diagram
\[
\begin{tikzcd}
  \H^2(W\mathscr{O}_X)\arrow{r}{d}\isor{d}{}&\dfrac{\H^2(W\Omega^1_X)}{\dlog\alpha}\isor{d}{}&\\
  \TR_{-2}(X,\alpha)\arrow{r}{d}&\TR_{-1}(X,\alpha)\arrow{r}{0}&\TR_0(X,\alpha)
\end{tikzcd}
\]
where the vertical arrows are induced by the descent spectral sequence, and the
right lower differential vanishes by Lemma \ref{lem:another little lemma}.
Finally, we have that $\TR_1(X,\alpha)=\TR_2(X,\alpha)=0$.

The differential $d$ is surjective, and we let $K(X,\alpha)$ denote its kernel, so that we have a short exact sequence
\[
  0\to K(X,\alpha)\to
  \H^2(W\mathscr{O}_X)\xrightarrow{d}\dfrac{\H^2(W\Omega^1_X)}{\dlog(\alpha)}\to
  0.
\]
We know that $K(X)=K(X,0)$ is a $k$-vector space of dimension $\sigma_0(X)$.
Recall from \cite[Corollary 3.4.23]{bragg-lieblich-twistor} that the Artin
invariant of a twisted supersingular K3 surface is given by
$\sigma_0(X,\alpha)=\sigma_0(X)+1$ if $\alpha\neq 0$ and
$\sigma_0(X,\alpha)=\sigma_0(X)$ otherwise. We conclude that $K(X,\alpha)$ is a
$k$-vector space of dimension $\sigma_0(X,\alpha)$. In particular, this shows
that the Artin invariant $\sigma_0(X,\alpha)$ is a derived invariant of
$(X,\alpha)$, which gives another proof of \cite[Corollary
3.4.3]{bragg-derived}.

The topological periodic cyclic homology $\TP(X,\alpha)$ is computed by the
Tate spectral sequence~\eqref{eq:Tate for TP}, whose $\E_2$ and $\E_3$ pages
are pictured in Figures~\ref{eq:Tate E2 for twisted TP} and~\ref{eq:Tate E3 for
twisted TP}.

\begin{figure}[H]
    \centering
    \begin{tikzcd}[row sep=small]
        \cdots&\TR_{0}(X,\alpha)            & 0      & \TR_{0}(X,\alpha) & 0      & \TR_{0}(X,\alpha)  & \cdots\\
        \cdots&\frac{\H^2(W\Omega^1_X)}{\dlog(\alpha)}\arrow{urr}{0} & 0      &
        \frac{\H^2(W\Omega^1_X)}{\dlog(\alpha)}\arrow{urr}{0}& 0 &
        \frac{\H^2(W\Omega^1_X)}{\dlog(\alpha)} & \cdots\\
        \cdots&\H^2(W\mathscr{O}_X)\arrow{urr}{d} & 0      & \H^2(W\mathscr{O}_X)\arrow{urr}{d}& 0      & \H^2(W\mathscr{O}_X)  & \cdots
    \end{tikzcd}
    \caption{A portion of the $\E_2$-page of the Tate spectral sequence for $(X,\alpha)$.}
\label{eq:Tate E2 for twisted TP}
\end{figure}
\begin{figure}[H]
    \centering
    \begin{tikzcd}[row sep=small]
        \cdots&\TR_{0}(X,\alpha)            & 0      & \TR_{0}(X,\alpha) & 0      & \TR_{0}(X,\alpha)  & \cdots\\
        \cdots& 0 & 0      & 0 & 0 & 0 & \cdots\\
        \cdots&K(X,\alpha) & 0      & K(X,\alpha)& 0      & K(X,\alpha)  & \cdots
    \end{tikzcd}
    \caption{A portion of the $\E_3=\E_{\infty}$-page of the Tate spectral sequence for $(X,\alpha)$.}
\label{eq:Tate E3 for twisted TP}
\end{figure}

We therefore find short exact sequences
\begin{equation}\label{eq:TP computation}
  0\to\TR_0(X,\alpha)\to\TP_i(X,\alpha)\to K(X,\alpha)\to 0
\end{equation}
of $W$-modules for all even $i$, and $\TP_i(X,\alpha)=\TP_i(X)=0$ for $i$ odd.
In particular, keeping track of the respective Frobenius actions as in the
previous section, we find for each $i\leq -1$ a short exact sequence
\begin{equation}\label{eq:TR and TP SES}
   0\to\TR_0(X,\alpha)(-i)\to\TP_{2i}(X,\alpha)\to K(X,\alpha)\to 0
\end{equation}
of $W$-modules, where the left hand arrow is a map of $F$-crystals.

If $\alpha=0$, then using~\eqref{eq:tricky twists 2} and~\eqref{eq:TR
computation} one checks that the inclusion $\TR_0(X)(-1)\to\TP_{-2}(X)$ above
is isomorphic to the inclusion of the Tate module of $\widetilde{\H}(X/W)$. The
cokernel of this inclusion is the same as the cokernel of the inclusion of the
Tate module of $\H^2(X/W)$. Thus, $K(X)$ is naturally identified with the
characteristic subspace associated to $X$ by Ogus \cite{MR563467}.

In general, one can show that for any $i\leq -1$ there is an isomorphism
\[
  \TP_{2i}(X,\alpha)\cong\widetilde{\H}(X/W,B)(i)
\]
of $F$-crystals, where $\widetilde{\H}(X/W,B)$ is the twisted K3 crystal
attached to $(X,\alpha)$ in \cite{bragg-derived}. As in the classical setting,
the construction of $\widetilde{\H}(X/W,B)$ depends on a non-canonical choice
of a $B$-field lift of $\alpha$, although the isomorphism class of the resulting
crystal is independent of this choice.
Furthermore, under this isomorphism, the inclusion
$\TR_0(X,\alpha)(-1)\to\TP_{-2}(X,\alpha)$ is identified with the inclusion of
the Tate module of $\widetilde{\H}(X/W,B)$, so that $K(X,\alpha)$ is identified
with the characteristic subspace associated to $(X,\alpha)$ defined in
\cite{bragg-derived}. In particular, as the dimension of $K(X,\alpha)$ is
determined by the $F$-crystal structure on $\TP_{2i}(X,\alpha)$, we see that if
$\alpha\neq 0$ then $\TP_{2i}(X)$ is not isomorphic to $\TP_{2i}(X,\alpha)$ as
an $F$-crystal for any $i\leq -1$, in contrast to the finite height case. We
remark that the derived Torelli theorem of \cite{bragg-derived} states that
$\Dscr^b(X,\alpha)$ is determined by the $F$-crystal $\widetilde{\H}(X/W,B)$
together with its Mukai pairing.

\addcontentsline{toc}{section}{References}

\small
\bibliographystyle{amsalpha}
\bibliography{main}

\vspace{20pt}
\scriptsize
\noindent
Benjamin Antieau\\
University of Illinois at Chicago\\
Department of Mathematics, Statistics, and Computer Science\\
851 South Morgan Street, Chicago, IL 60607\\
\texttt{benjamin.antieau@gmail.com}

\vspace{10pt}
\noindent
Daniel Bragg\\
UC Berkeley\\
Department of Mathematics\\
970 Evans Hall, Berkeley, CA 94720\\
\texttt{braggdan@math.berkeley.edu}

\end{document}